\documentclass[sn-mathphys-num]{sn-jnl}

\usepackage{graphicx}%
\usepackage{multirow}%
\usepackage{amsmath,amssymb,amsfonts,dsfont}%
\usepackage{amsthm}%
\usepackage{mathrsfs}%
\usepackage[title]{appendix}%
\usepackage{xcolor}%
\usepackage{textcomp}%
\usepackage{manyfoot}%
\usepackage{booktabs}%
\usepackage{algorithm}%
\usepackage{algorithmicx}%
\usepackage{algpseudocode}%
\usepackage{listings}%

\theoremstyle{thmstyleone}%
\newtheorem{theorem}{Theorem}
\newtheorem{proposition}[theorem]{Proposition}%
\newtheorem{corollary}[theorem]{Corollary}%

\theoremstyle{thmstyletwo}%
\newtheorem{remark}{Remark}%

\theoremstyle{thmstylethree}%
\newtheorem{definition}{Definition}%

\begin{document}

\title[Testing Hooke-like isotropic hyper-/hypo-elastic material models under finite simple shear deformations]{Testing Hooke-like isotropic hyper-/hypo-elastic material models under finite simple shear deformations}


\author*[1]{\fnm{Sergey N.} \sur{Korobeynikov}}\email{s.n.korobeynikov@mail.ru}

\author[1]{\fnm{Alexey Yu.} \sur{Larichkin}}\email{larichking@gmail.com}

\author[2]{\fnm{Patrizio} \sur{Neff}}\email{patrizio.neff@uni-due.de}

\affil*[1]{\orgname{Lavrentyev Institute of Hydrodynamics SB RAS}, \orgaddress{\street{Lavrentyev av., 15}, \city{Novosibirsk}, \postcode{630090}, \country{Russia}}}

\affil[2]{\orgdiv{Chair for Nonlinear Analysis and Modelling, Faculty of Mathematics}, \orgname{University of Duisburg-Essen}, \orgaddress{\street{Thea-Leymann Strasse, 9}, \city{Essen}, \postcode{D-45127}, \country{Germany}}}



\abstract{We test some Hooke-like isotropic hyper-/hypo-elastic material models under finite simple shear deformations (cf., Thiel et al. Int. J. Non-linear Mech. 112: 57--72,  2019) and show that (1) the components of the Cauchy stress tensor for any Cauchy/Green isotropic elastic material under left finite simple shear (LFSS) deformation are equal to the components of the rotated Cauchy stress tensor for the same material under right finite simple shear (RFSS) deformation; (2) for any Hill's linear isotropic hyperelastic material model based on a symmetrically physical (SP) strain measure, LFSS and RFSS deformations lead to Eulerian and Lagrangian pure shear stresses, respectively; (3) for any two-power Ogden's isotropic hyperelastic material model based on a SP strain function, LFSS and RFSS deformations lead to Eulerian and Lagrangian pure shear stresses, respectively; (4) for some Hooke-like isotropic hypoelastic materials with constitutive relations based on corotational stress rates under LFSS deformation, the behavior of the Cauchy stress tensor components as a function of the shear parameter is qualitatively similar to that for the same materials under simple shear deformation. In addition, we confirm the results of Lin (Lin R.C. ZAMP, 75: 191, 2024) showing that for some Hooke-like isotropic hypoelastic materials with constitutive relations based on corotational stress rates without initial stresses under RFSS deformation, the Cauchy stress tensor components coincide with those for the Hencky isotropic hyperelastic material.}




\maketitle

\section{Introduction}
\label{sec:1}

We consider a homogeneous deformation (no matter now whether it is infinitesimal or finite) of a sample of an isotropic material for which the Cauchy stress tensor components are expressed in a Cartesian coordinate system with orthonormal basis vectors $\mathbf{k}_1$, $\mathbf{k}_2$, and $\mathbf{k}_3$ and in matrix form as follows:
\begin{equation}\label{1-1}
   \boldsymbol{\sigma}=\sigma_{12}\mathbf{k}_1\otimes \mathbf{k}_2 + \sigma_{12}\mathbf{k}_2\otimes \mathbf{k}_1\quad \Leftrightarrow \quad
   \boldsymbol{\sigma} = \left[                                                                                                                                                                   \begin{array}{ccc}
      0           & \sigma_{12} & 0 \\
      \sigma_{12} & 0           & 0 \\
      0           & 0           & 0 \\
   \end{array}
   \right].
\end{equation}
The eigenvalues (principal stresses) and subordinate orthonormal eigenvectors of the stress tensor $\boldsymbol{\sigma}$ have the following form ($s\equiv \sigma_{12}=\sigma_{21}$):
\begin{equation*}
  \sigma_1=s,\quad  \sigma_2=-s,\quad \sigma_3=0,\quad \mathbf{n}_1=\frac{\sqrt{2}}{2}\left[
                                                                      \begin{array}{c}
                                                                        1 \\
                                                                        1 \\
                                                                        0 \\
                                                                      \end{array}
                                                                    \right],\quad
                                            \mathbf{n}_2=\frac{\sqrt{2}}{2}\left[
                                                                      \begin{array}{c}
                                                                       -1 \\
                                                                        1 \\
                                                                        0 \\
                                                                      \end{array}
                                                                    \right],\quad
                                            \mathbf{n}_3=\left[
                                                                      \begin{array}{c}
                                                                       0 \\
                                                                       0 \\
                                                                       1 \\
                                                                      \end{array}
                                                                    \right].
\end{equation*}
This stress tensor state is called the \emph{pure shear stress} state (see Fig. \ref{f1} for illustration).
\begin{figure}
\begin{center}
\includegraphics{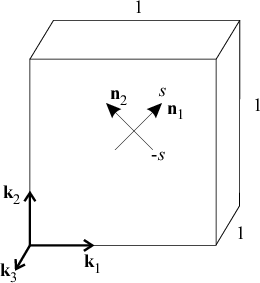}
\end{center}
\caption{Pure shear stress state in the homogeneous deformable sample in the current configuration.}
\label{f1}
\end{figure}

For infinitesimal deformations of linear isotropic elastic materials, the Cauchy stress tensor \eqref{1-1} corresponds to isochoric deformation ($\text{tr}\,\boldsymbol{\varepsilon}=0$)\footnote{As noted in \cite{ThielIJNLM2019},  this condition is volume preserving only to first-order accuracy.} with the infinitesimal strain tensor $\boldsymbol{\varepsilon}$ given by the formula \eqref{1-1} in which the component $\sigma_{12}$ of the tensor $\boldsymbol{\sigma}$ is replaced by the component $\varepsilon_{12}$ of the tensor $\boldsymbol{\varepsilon}$. Deformation of this type (\emph{pure shear stretch}) is illustrated in Fig. \ref{f2},\emph{a}
\begin{figure}
\begin{center}
\includegraphics{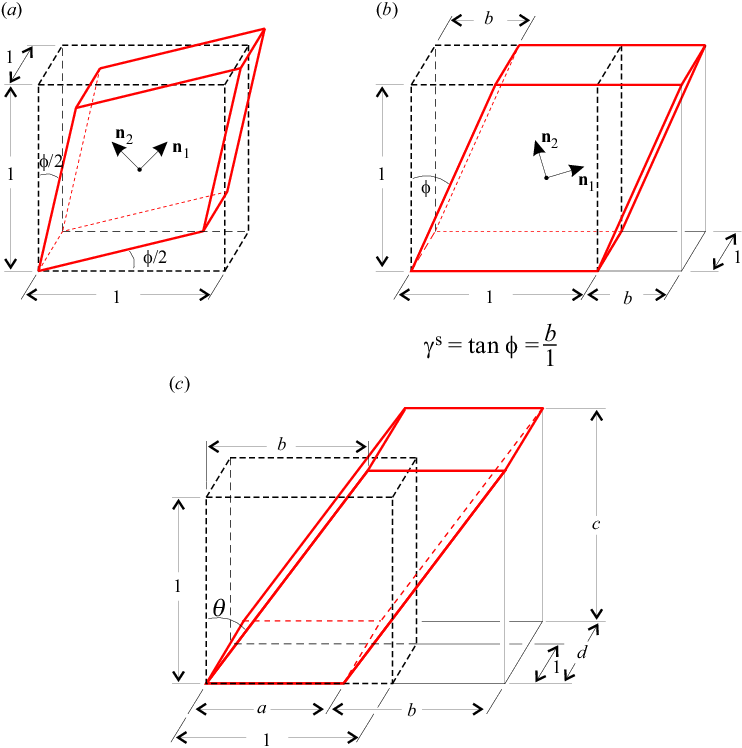}
\end{center}
\caption{Different types of shear deformation: (\emph{a}) pure shear stretch for infinitesimal strains; (\emph{b}) simple shear; (\emph{c}) pure shear stretch for finite strains.}
\label{f2}
\end{figure}
The constitutive relations for this material reduce to the form
\begin{equation}\label{1-3}
  \sigma_{12}=2 \mu\, \varepsilon_{12},
\end{equation}
where $\mu>0$ is the shear modulus.

Rivlin \cite{RivlinPRSLA1948} (see also \cite{BoulangerARMA2000,DestradeIJNLM2012,DienesAM1979,DienesAM1987,HorganJElast2010,GambirasioACME2016,KorobeynikovAAM2020,LinIJNME2002,LinZAMM2003,LinEJMAS2003,LiuJElast1999,MihaiPRSA2011,MihaiIJNLM2013,Prager1961,SzaboIJSS1989,XiaoZAMM2006}) introduced the definition of \emph{simple shear} deformation with kinematics suitable for studying the shear deformation of solids with the law of motion $\mathbf{x}(\mathbf{X})$ of the form (see Fig. \ref{f2},\emph{b})
\begin{equation}\label{1-4}
    x_1(t)=X_1+\gamma(t)X_2,\quad\quad x_2(t)=X_2,\quad\quad x_3(t)=X_3\quad\quad (t\in [t_0,\infty]).
\end{equation}
Here $X_i$ and $x_i$ ($i=1,2,3$) are the Cartesian coordinates of any material point of the solid in the reference and current configurations, respectively, and the shear  parameter $\gamma(t)$ is a given monotonically increasing function of time $t$ ($\gamma(t_0)=0$). We are interested in the evolution of the Cauchy stress tensor components $\sigma_{11}$, $\sigma_{22}$, and $\sigma_{12}$ as a function of the parameter $\gamma(t)$ which characterizes the amount of shear. Many researchers have used the law of motion \eqref{1-4} to test different materials. For infinitesimal deformations, the laws of motion presented in Figs. \ref{f2},\emph{a} and \ref{f2},\emph{b} are equivalent, i.e., the law of motion \eqref{1-4} for isotropic solids in linear elasticity leads to constitutive relations \eqref{1-3} \cite{ThielIJNLM2019}, since the laws of motion in Figs. \ref{f2},\emph{a} and \ref{f2},\emph{b} differ only by rotation and the equations of linear elasticity neglect rotations. Therefore, in linear elasticity theory, the law of motion \eqref{1-4} for isotropic solids corresponds to a pure shear stress of the form \eqref{1-1}. However, in finite (Cauchy/Green) elasticity theory, the simple shear deformation is defined by the universal relation \cite{RivlinPRSLA1948}
\begin{equation*}
  \gamma\, \sigma_{12}=\sigma_{11} -\sigma_{22},
\end{equation*}
whence a simple shear deformation cannot lead to a pure shear stress state for any Cauchy/Green isotropic elastic materials.

Moon and Truesdell \cite{MoonARMA1974} asked the question: is it possible in the case of deformation with the law of motion (see Fig. \ref{f2},\emph{c})
\begin{equation}\label{1-6}
    x_1(t)=a(t)X_1+b(t)X_2,\quad x_2(t)=c(t)X_2,\quad x_3(t)=d(t)X_3\quad t\in [t_0,\infty]
\end{equation}
to find expressions for $a$, $b$, $c$, and $d$ such that Cauchy/Green isotropic elastic solids experience a pure shear stress state of the form \eqref{1-1}? It has been shown \cite{MoonARMA1974} (see also \cite{BelikJElast1998,BoulangerJElast2004,DestradeIJNLM2012,HayesJElast2008,HorganJElast2023a,LinZAMP2024,MihaiPRSA2011,MihaiIJNLM2013,MurphyIJNLM2022,RubinJElast1988,RubinJElast2020,ThielIJNLM2019,TingJElast2006}) that a necessary condition for this is the satisfaction of the equality
\begin{equation}\label{1-7}
  a^2 + b^2 = c^2.
\end{equation}
In addition, it has been shown \cite{MihaiPRSA2011,MoonARMA1974} that, when constitutive relations are subject to an additional constraint in the form of the Baker--Ericksen inequalities, equality \eqref{1-7} is also a sufficient condition for the existence of the pure shear stress state \eqref{1-1} for samples of the indicated materials under homogeneous deformation with the law of motion \eqref{1-6} for some values of the parameters $a$, $b$, $c$, and $d$. Furthermore, those studies have shown that the deformation of a sample with the law of motion \eqref{1-6} corresponds to a simple shear superimposed on a triaxial stretch. In this case, the deformation may not be isochoric, i.e., the Kelvin effect may occur resulting in a volume change of the sample in the pure shear stress with an escape from the plane $(X_1,X_2)$, i.e., for $d\neq 1$.

Thiel et al. \cite{ThielIJNLM2019} argued that the single constraint \eqref{1-7} on the quantities $a$, $b$, $c$, and $d$ greatly distances the law of motion \eqref{1-6} from the law of motion \eqref{1-4} for the simple shear deformation that in isotropic linear elasticity  leads to the simple shear stress state \eqref{1-1}. In addition to \eqref{1-7}, they proposed two additional constraints on the law of motion \eqref{1-6}: $d(t)= 1$ (i.e., the requirement of plane strain deformation) and $J\,(\equiv \det\partial \mathbf{x}/\partial \mathbf{X})=1$ (i.e., the requirement of isochoric deformation). These authors defined homogeneous deformations corresponding to the law of motion \eqref{1-6} with the above-mentioned constraints on the quantities $a$, $b$, $c$, and $d$ as the \emph{left finite simple shear} (\emph{LFSS}) deformations. In addition, they introduced another type of homogeneous deformation corresponding to the law of motion \eqref{1-6} with the constraints $d(t)= 1$ and $J=1$ on $a$, $b$, $c$, and $d$, but replaced the constraint \eqref{1-7} by the equality
\begin{equation}\label{1-8}
  a^2 = b^2 + c^2,
\end{equation}
and called this kind of deformation corresponding to the law of motion \eqref{1-6} with the above three constraints on $a$, $b$, $c$, and $d$ the \emph{right finite simple shear} (\emph{RFSS}) deformation.

Thiel et al. \cite{ThielIJNLM2019} have demonstrated that LFSS deformations correspond to pure shear stress not for all Cauchy/Green elastic isotropic materials. In particular, they have shown that, for neo-Hookean isotropic hyperelastic materials, LFSS deformations do not correspond to pure shear stress. This result is consistent with the statement of Horgan and Murphy \cite{HorganJElast2023a} (p. 88) that for any one-power isotropic hyperelastic Ogden's model (in particular, the neo-Hookean hyperelastic isotropic material model), except for the Hencky model, a simple shear stress corresponds to a deformation with the law of motion \eqref{1-6} subject to \eqref{1-7} only for $d>1$. Thiel et al. \cite{ThielIJNLM2019} have also shown that there exist Cauchy/Green elastic isotropic materials for which LFSS and RFSS deformations correspond to pure shear stresses (for RFSS deformations, the Eulerian Cauchy stress tensor is taken to be its rotated counterpart). They have shown, in particular, that these materials include the Hencky isotropic hyperelastic material (see also Eq. (3.9) in \cite{HorganJElast2023a}).

LFSS and RFSS deformations are of course suitable for testing not only Cauchy/Green elastic material models, but also all other material models. For example, Lin \cite{LinZAMP2024} have tested some Hooke-like isotropic hypoelastic material models based on corotational stress rates under RFSS deformation and found that this kind of deformation leads to the same Lagrangian pure stress state for all these models.

The present study has the following main objectives:
\begin{enumerate}
  \item to test two one-power Ogden's material models under LFSS/RFSS deformations (one of these models coincides with the neo-Hookean model for incompressible materials \cite{Bertram2021} and with the Simo--Pister model for compressible materials \cite{SimoCMAME1984}, see also \cite{Korobeynikov2026});
  \item to identify a subfamily of Hill's linear isotropic hyperelastic (HLIH) material models \cite{Hill1979} (see also \cite{KorobeynikovJElast2019,KorobeynikovIJSS2022,KorobeynikovMTDM2024}) for which LFSS and RFSS deformations correspond to Eulerian and Lagrangian pure stress states, respectively;
  \item to identify a subfamily of two-power Ogden's isotropic hyperelastic material models \cite{OgdenPRSLA1972a,OgdenPRSLA1972b} (see also \cite{KorobeynikovAAM2023,KorobeynikovAAM2025}) for which LFSS and RFSS deformations correspond to Eulerian and Lagrangian pure stress states, respectively;
  \item to test some Hooke-like isotropic hypoelastic material models based on corotational stress rates under LFSS/RFSS deformation.
\end{enumerate}

The first main new result of this study is the identification of subfamilies of two well-known (Hill's and Ogden's) families of isotropic hyperelasticity models for which LFSS and RFSS deformations correspond to Eulerian and Lagrangian pure stress states, respectively. The second result is testing some well-known Hooke-like isotropic hypoelastic material models not only under RFSS deformation, as was done in \cite{LinZAMP2024}, but also under LFSS deformation, which is a new result in the testing of such material models under homogeneous deformations. Incidentally, we note that our analytical solutions are obtained in a simpler way than in \cite{LinZAMP2024}. In addition, we consider some misleading formulations of constitutive relations for material models in \cite{LinZAMP2024}. Other new results of this work include finding spurious oscillations of stresses \emph{under LFSS deformation} for the Hooke-like (grade zero) hypoelastic material model (with constant fourth order stiffness tensor) based on the corotational Zaremba--Jaumann stress rate (which is consistent with the similar behavior of the solution for the simple shear problem, see, e.g., \cite{KorobeynikovAAM2020,Korobeynikov2023,KorobeynikovZAMM2024}), the  absence of such oscillations for the Hooke-like (grade zero) hypoelastic material model based on the corotational logarithmic stress rate without initial stresses, and the presence of spurious oscillations of stresses for this model under LFSS deformation with non-zero initial .stresses (which is consistent with the similar behavior of stresses for this material model under simple shear deformation, see, e.g., \cite{LiuJElast1999,KorobeynikovZAMM2024}). Note also that for greater reliability of the results, we confirmed some analytical solutions of the problems under consideration by computer simulations of the tested material models implemented in both the commercial MSC.Marc \cite{MarcA2015} nonlinear FE system and the homemade Pioner \cite{Korobeinikov1989} FE code.

\section{Preliminaries}
\label{sec:2}

This section provides a background for the formulation of constitutive relations for hyper-/hypo-elastic materials. Operations on tensors and the spectral representation of a symmetric second-order tensor are presented in Section \ref{sec:2-1}. Section \ref{sec:2-2} provides formulations of local body deformations and basic kinematics that will be used in Section \ref{sec:2-3} to derive constitutive relations for Hooke-like isotropic hyper-/hypo-elastic material models.

\subsection{Operations on tensors and spectral representation of a symmetric second-order tensor}
\label{sec:2-1}

We define the second-order tensors $\mathbf{H}\in \mathcal{T}^2$ (hereinafter, $\mathcal{T}^2$ denotes the set of all second-order tensors). Hereinafter, $\mathcal{T}^2_{\text{sym}},\  \mathcal{T}^2_{\text{skew}}\subset \mathcal{T}^2$ denote the sets of all symmetric and skew-symmetric second-order tensors, respectively; $\text{sym}\,\mathbf{A}\equiv (\mathbf{A}+\mathbf{A}^T)/2\in \mathcal{T}^2_{\text{sym}}$ and $\text{skew}\,\mathbf{A}\equiv (\mathbf{A}-\mathbf{A}^T)/2\in \mathcal{T}^2_{\text{skew}}$ denote the symmetric and skew-symmetric components, respectively, of the tensor $\mathbf{A}\in \mathcal{T}^2$.

Let $\mathbf{A},\mathbf{H}\in \mathcal{T}^2$. We define the \emph{double inner product} (\emph{double contraction}) operation of tensors:
\begin{equation*}
    \mathbf{A}:\mathbf{H}=\mathbf{H}:\mathbf{A} \equiv \text{tr}(\mathbf{A}\cdot\mathbf{H}^T)=  \text{tr}(\mathbf{A}^T\cdot\mathbf{H})=
    \text{tr}(\mathbf{H}\cdot\mathbf{A}^T)=\text{tr}(\mathbf{H}^T\cdot\mathbf{A}).
\end{equation*}
Hereinafter, the superscript $T$ denotes the transposition of a tensor, and the dot between vectors and/or tensors denotes their inner product.

Let $\mathbf{S}\in \mathcal{T}^2_{\text{sym}}$. This tensor can be represented in the classical spectral form
\begin{equation*}
\mathbf{S}=\sum_{k=1}^{3} s_k\, \mathbf{s}_k\otimes\mathbf{s}_k,
\end{equation*}
where $s_k\in \mathds{R}$ are eigenvalues (hereinafter, the subscript $k$ runs over the values 1, 2, and 3), $\{\mathbf{s}_k\}$ ($k=1,2,3$) is the triad of the corresponding subordinate right-oriented orthonormal eigenvectors of the tensor $\mathbf{S}$, and the symbol $\otimes$ denotes the dyadic product of vectors. For multiple eigenvalues $s_k$, the corresponding eigenvectors $\mathbf{s}_k$ are defined ambiguously. This ambiguity can be circumvented by representing the tensor $\mathbf{S}$ in terms of \emph{eigenprojections} (see, e.g., \cite{Bertram2021,Hashiguchi2013,Itskov2019,KorobeynikovAM2011,KorobeynikovAM2023,LuehrCMAME1990,XiaoIJSS1995}) as
\begin{equation*}
    \mathbf{S}=\sum^m_{i=1}s_i\, \mathbf{S}_i.
\end{equation*}
Here $s_i$ are all different $m$\footnote{The number $m$ ($1\leq m\leq 3$) will be called the \emph{eigenindex}.} eigenvalues of the tensor $\mathbf{S}$ and $\mathbf{S}_i$ are the subordinate \emph{eigenprojections} (hereinafter, the subscripts run from 1 to $m$). Hereinafter, we introduce the eigenprojections of the tensor $\mathbf{S}$ using Sylvester's formula
\begin{equation}\label{2-3}
    \mathbf{S}_i=
    \begin{cases}
    \prod^m\limits_{\substack{j=1\\ i\neq j}}  \frac{\mathbf{S}-s_j\mathbf{I}}{s_i-s_j} & \text{if}\ m=2,3 \\
    \mathbf{I} & \text{if}\ m=1
   \end{cases}\ ,
\end{equation}
hereinafter, $\mathbf{I}$ denotes the second order identity tensor. The eigenprojections have the following properties (see, e.g., \cite{LuehrCMAME1990}):
\begin{equation*}
   \mathbf{S}_i \cdot \mathbf{S}_j=
   \begin{cases}
    \mathbf{S}_i, & \text{if}\ i=j, \\
    \mathbf{O}, & \text{if}\ i\neq j,
   \end{cases}
\quad\quad \sum^m_{i=1}\mathbf{S}_i=\mathbf{I}, \quad \text{tr}\,\mathbf{S}_i=m_i.
\end{equation*}
Here $m_i$ denotes the multiplicity of an eigenvalue $s_i$ and $\mathbf{O}\in \mathcal{T}^2$ denotes the zero second-order tensor. The definition of the eigenprojections \eqref{2-3} implies their symmetry, i.e., $\mathbf{S}_i\in \mathcal{T}^2_{\text{sym}}$ ($\mathbf{S}_i^T=\mathbf{S}_i$).\\

\subsection{Local body deformations and basic kinematics}
\label{sec:2-2}

Consider the motion of a body $\mathfrak{B}$ in a three-dimensional Euclidean point space, and let $\mathbf{X}$ and $\mathbf{x}$ be the position vectors of some particle $P\in \mathfrak{B}$ in the \emph{reference} (fixed at time $t_0$) and \emph{current} (moving at time $t$) configurations, respectively. Let $\mathbf{F}\equiv \text{Grad}\,\mathbf{x}= \partial \mathbf{x}/\partial \mathbf{X}\in \mathcal{T}^2$ ($J\equiv\det \mathbf{F} >0$) be the \emph{deformation gradient}. The tensor $\mathbf{F}$ can be uniquely represented as (see, e.g., \cite{Bertram2021,Itskov2019,NeffIJES2014,Ogden1984,Truesdell1965})
\begin{equation}\label{2-5}
\mathbf{F} = \mathbf{R}\cdot\mathbf{U} = \mathbf{V}\cdot\mathbf{R} \quad (\mathbf{U}^T=\mathbf{U},\quad \mathbf{V}^T=\mathbf{V},\quad \mathbf{R}\cdot\mathbf{R}^T=\mathbf{I},\quad \det\mathbf{R}=1),
\end{equation}
where $\mathbf{U}=\sqrt{\mathbf{F}^T \mathbf{F}},\ \mathbf{V}=\sqrt{\mathbf{F}\mathbf{F}^T}\in\mathcal{T}_{\text{sym}}^2$ are the \emph{right} (Lagrangian) and \emph{left} (Eulerian) positive definite \emph{stretch tensors}, respectively, and $\mathbf{R}\in \mathcal{T}^{2\,+}_\text{orth}$ is the \emph{rotation tensor}.\footnote{Hereinafter, the subset $\mathcal{T}^{2\,+}_\text{orth}\subset \mathcal{T}^2$ denotes the set of all proper orthogonal second-order tensors (i.e., the tensors $\boldsymbol{\Psi}$ such that $\boldsymbol{\Psi}\cdot\boldsymbol{\Psi}^T=\mathbf{I}$ and $\det \boldsymbol{\Psi}=1$).}

Consider the \emph{Euclidean transformations} (ETs) (cf., \cite{KorobeynikovJElast2008,Ogden1984})
\begin{equation}\label{2-6}
\mathbf{x}^{\ast}(\mathbf{X}^{\ast},t^{\ast})\equiv \mathbf{Q}(t)\cdot\mathbf{x} (\mathbf{X},t)+\mathbf{c}(t), \quad
\mathbf{X}^{\ast}\equiv\mathbf{Q}_0\cdot\mathbf{X}+\mathbf{c}_0, \quad t^{\ast}=t+a,
\end{equation}
where $\mathbf{Q}(t)\in\mathcal{T}^{2\,+}_{\text{orth}}$ is an arbitrary tensor, $\mathbf{c}(t)$ is an arbitrary vector, and $a\in\mathds{R}$ is an arbitrary variable ($\mathbf{Q}_0\equiv\mathbf{Q}(t_0)$ and $\mathbf{c}_0\equiv\mathbf{c}(t_0)$). We now give definitions of objective tensors following \cite{Holzapfel2000,KorobeynikovJElast2008,Ogden1984}.\\

\begin{definition}
\label{def:2-1}
A tensor $\mathbf{H}\in\mathcal{T}^2$ is called an \emph{objective tensor} if, under ETs \eqref{2-6}, it changes as follows:\footnote{Further, objective Eulerian and Lagrangian tensors will be called for brevity Eulerian and Lagrangian tensors.}
\begin{alignat*}{3} 
    \mathbf{H}^{\ast}(P,t) &= \mathbf{Q}(t) \cdot \mathbf{H}(P,t) \cdot \mathbf{Q}^T(t)& \quad & \text{\emph{Eulerian}}, \\
    \mathbf{H}^{\ast}(P,t) &= \mathbf{Q}_0 \cdot \mathbf{H}(P,t) \cdot \mathbf{Q}_0^T &  \quad & \text{\emph{Lagrangian}}. \notag
\end{alignat*}
\end{definition}

Let the Lagrangian and Eulerian tensors $\mathbf{U}$ and $\mathbf{V}$ have the eigenindex $m$. The spectral representations of the tensors $\mathbf{U}$ and $\mathbf{V}$ have the form
\begin{equation*}
    \mathbf{U}=\sum_{k=1}^{3} \lambda_k\,\mathbf{N}_k\otimes\mathbf{N}_k=\sum^m_{i=1}\lambda_i\, \mathbf{U}_i,\quad
    \mathbf{V}=\sum_{k=1}^{3} \lambda_k\,\mathbf{n}_k\otimes\mathbf{n}_k=\sum^m_{i=1}\lambda_i\, \mathbf{V}_i\quad (0<\lambda_i<\infty).
\end{equation*}
Hereinafter, $\lambda_k$ are eigenvalues (principal stretches) of the tensors $\mathbf{U}$ and $\mathbf{V}$, $\mathbf{N}_k$ and $\mathbf{n}_k$ ($\mathbf{n}_k=\mathbf{R}\cdot \mathbf{N}_k$) ($k=1,2,3$) are the subordinate eigenvectors (corresponding to the principal directions), and $\mathbf{U}_i$ and $\mathbf{V}_i$ ($i=1,\ldots,m$) are the subordinate eigenprojections of these tensors.

For applications, it is possible to define (see, e.g., \cite{CurnierET1991}) the \emph{right and left Cauchy--Green deformation tensors}
\begin{align}\label{2-9}
   (\mathbf{F}^T \mathbf{F}=)\,\mathbf{C}\equiv \mathbf{U}^2 = &\sum_{k=1}^{3} \lambda^2_k\,\mathbf{N}_k\otimes\mathbf{N}_k=\sum^m_{i=1}\lambda^2_i\, \mathbf{U}_i,  \\
   (\mathbf{F} \mathbf{F}^T=\,\mathbf{B}=)\, \mathbf{c}\equiv \mathbf{V}^2 = &\sum_{k=1}^{3} \lambda^2_k\,\mathbf{n}_k\otimes\mathbf{n}_k=\sum^m_{i=1}\lambda^2_i\, \mathbf{V}_i. \notag
\end{align}
and their inverse tensors
\begin{align}\label{2-10}
    \mathbf{C}^{-1} &= \mathbf{U}^{-2} = \sum_{k=1}^{3} \lambda^{-2}_k\,\mathbf{N}_k\otimes\mathbf{N}_k=\sum^m_{i=1}\lambda^{-2}_i\, \mathbf{U}_i, \\
    (\mathbf{B}^{-1} =)\,\mathbf{c}^{-1} &= \mathbf{V}^{-2} = \sum_{k=1}^{3} \lambda^{-2}_k\,\mathbf{n}_k\otimes\mathbf{n}_k=\sum^m_{i=1}\lambda^{-2}_i\, \mathbf{V}_i.  \notag
\end{align}

The simple \cite{CurnierET1991} isotropic tensor functions
\begin{align*}
\mathbf{E} &\equiv \mathbf{f}(\mathbf{U})=\sum_{k=1}^{3} f(\lambda_k)\,\mathbf{N}_k\otimes\mathbf{N}_k=\sum^m_{i=1}f(\lambda_i)\, \mathbf{U}_i, \\
\mathbf{e} &\equiv \mathbf{f}(\mathbf{V})=\sum_{k=1}^{3} f(\lambda_k)\,\mathbf{n}_k\otimes\mathbf{n}_k=\sum^m_{i=1}f(\lambda_i)\, \mathbf{V}_i, \notag \\
& f(\lambda)\in\mathds{R}:\ f(\lambda)\in C^2, \quad  f^{\prime}(\lambda)>0 \ \text{as}\ \lambda\in(0,\infty)\ \text{and}\ f(1)=0,\ f^{\prime}(1)=1 \notag
\end{align*}
form, respectively, the Lagrangian and Eulerian subfamilies of the \emph{Hill family of strain tensors} (cf.,  \cite{Hill1979,Ogden1984}).

Table \ref{t1} gives some strain tensors $\mathbf{E}$ and $\mathbf{e}$ belonging to the Hill family.
\begin{table}
\caption{Some strain tensors belonging to the Hill family}
\label{t1}
\begin{tabular}{llll}
\hline\noalign{\smallskip}
 Generating      & Objectivity & Basis-free & Strain     \\
 scale function & type        & expression & tensor name \\
\noalign{\smallskip}\hline\noalign{\smallskip}
   & Lagrangian & $\mathbf{E}^{(2)}\equiv \frac{1}{2}(\mathbf{U}^2-\mathbf{I})$   & Green--Lagrange \\
 $f^{(2)}(\lambda)\equiv \frac{1}{2}(\lambda^2-1)$  & Eulerian  & $\mathbf{e}^{(2)} \equiv \frac{1}{2}(\mathbf{V}^2-\mathbf{I})$ & Finger \\
\noalign{\smallskip}\hline\noalign{\smallskip}
   & Lagrangian & $\mathbf{E}^{(1)}\equiv \mathbf{U}-\mathbf{I}$   & right Biot \\
 $f^{(1)}(\lambda)\equiv \lambda-1 $  & Eulerian  & $\mathbf{e}^{(1)} \equiv \mathbf{V} - \mathbf{I}$   & left Biot  \\
\noalign{\smallskip}\hline\noalign{\smallskip}
   & Lagrangian & $\mathbf{E}^{(0)}\equiv \log \mathbf{U}$   & right Hencky \\
 $f^{(0)}(\lambda)\equiv \log \lambda$  & Eulerian  & $\mathbf{e}^{(0)} \equiv \log\mathbf{V}$   & left Hencky  \\
\noalign{\smallskip}\hline\noalign{\smallskip}
   & Lagrangian & $\mathbf{E}^{(-1)}\equiv \mathbf{I}-\mathbf{U}^{-1}$   & Hill \\
 $f^{(-1)}(\lambda)\equiv 1-\lambda^{-1}$  & Eulerian  & $\mathbf{e}^{(-1)} \equiv \mathbf{I}-\mathbf{V}^{-1}$   & Swainger  \\
\noalign{\smallskip}\hline\noalign{\smallskip}
   & Lagrangian & $\mathbf{E}^{(-2)}\equiv\frac{1}{2}(\mathbf{I}-\mathbf{U}^{-2})$   & Karni--Reiner \\
 $f^{(-2)}(\lambda)\equiv \frac{1}{2}(1-\lambda^{-2})$  & Eulerian  & $\mathbf{e}^{(-2)} \equiv \frac{1}{2}(\mathbf{I}-\mathbf{V}^{-2})$   & Almansi  \\
\noalign{\smallskip}\hline\noalign{\smallskip}
   & Lagrangian & $\mathbf{E}^M\equiv \frac{1}{4}(\mathbf{U}^2-\mathbf{U}^{-2}) $   & right Mooney \\
 $f^M(\lambda)\equiv \frac{1}{4}(\lambda^2-\lambda^{-2})$  & Eulerian  & $\mathbf{e}^M \equiv \frac{1}{4}(\mathbf{V}^2-\mathbf{V}^{-2})$   & left Mooney  \\
\noalign{\smallskip}\hline\noalign{\smallskip}
   & Lagrangian & $\mathbf{E}^P\equiv \frac{1}{2}(\mathbf{U}-\mathbf{U}^{-1})$   & right Pelzer \\
 $f^P(\lambda)\equiv \frac{1}{2}(\lambda-\lambda^{-1})$  & Eulerian  & $\mathbf{e}^P \equiv \frac{1}{2}(\mathbf{V}-\mathbf{V}^{-1})$   & left Pelzer  \\
\noalign{\smallskip}\hline
\end{tabular}
\end{table}
Note that all strain tensors given in Table \ref{t1} belong to the two-parameter Curnier--Rakotomanana \cite{CurnierET1991} and the two-parameter power \cite{DarijaniIJES2010,DarijaniIMechE2010} strain tensor families. The strain tensors generated by the scale functions $f^{(2)}(\lambda)$, $f^{(-2)}(\lambda)$, $f^{(1)}(\lambda)$, $f^{(-1)}(\lambda)$, and $f^{(0)}(\lambda)$ also belong to the one-parameter Doyle--Ericksen \cite{Doyle1956} strain tensor family. In addition, the Hencky, Mooney, and Pelzer strain tensors also belong to the one-parameter Ba\v{z}ant--Itskov \cite{BazantJEMT1998,Itskov2019,ItskovMRC2004} and the symmetrically physical \cite{KorobeynikovJElast2019} strain tensor families.

It can be shown that the Mooney $\mathbf{E}^M$ and $\mathbf{e}^M$ and Pelzer $\mathbf{E}^P$ and $\mathbf{e}^P$ strain tensors approximate the Hencky strain tensors $\log\mathbf{U}$ and $\log\mathbf{V}$ up to third-order terms, and the tensors $\mathbf{E}^{(2)}$, $\mathbf{E}^{(-2)}$, $\mathbf{e}^{(2)}$, and $\mathbf{e}^{(-2)}$ approximate these tensors only up to second-order terms. Plots of the scale functions generating these strain tensors versus $\lambda$ are given in Fig. \ref{f3}.
\begin{figure}
\begin{center}
\includegraphics{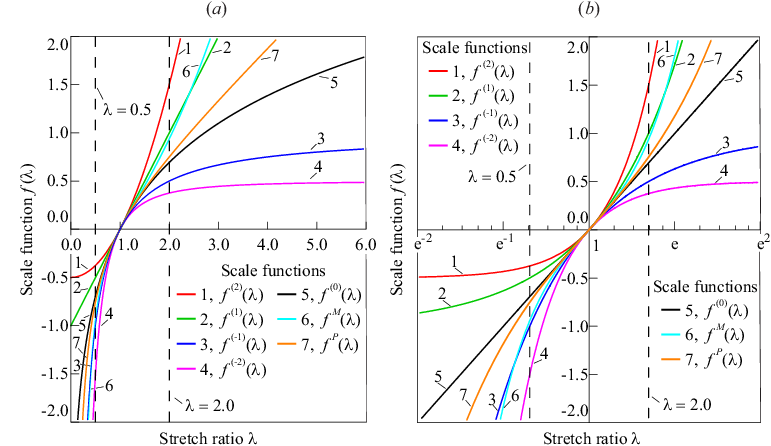}
\end{center}
\caption{Plots of the scale functions $f(\lambda)$ generating the strain tensors given in Table \ref{t1} with (\emph{a}) standard and (\emph{b}) logarithmic scales for the stretch $\lambda$.}
\label{f3}
\end{figure}

We now give expressions for two families of scale functions (both families are subfamilies of scale functions generating strain tensors from the Hill family) used in this paper:
\begin{description}
  \item[(i)] The one-parameter \emph{Doyle--Ericksen family} (with the parameter $n\in\mathds{R}$) \emph{of strain tensors} \cite{Doyle1956} is formed using a function $f(\lambda)$ of the form\footnote{Some authors call this family of strain tensors \emph{the Seth--Hill family of strain tensors} (see, e.g., \cite{CurnierET1991}).}
\begin{equation*}
f^{(n)}(\lambda) \equiv
\begin{cases}
\frac{1}{n} (\lambda^n-1), &\text{if}\  n \neq 0, \\
\log\lambda\ [=\lim \limits_{n \to 0} \frac{1}{n} (\lambda^n-1) ],  &\text{if} \  n=0.
\end{cases}
\end{equation*}
  \item[(ii)] The one-parameter \emph{Ba\v{z}ant--Itskov family} (with the parameter $r\in\mathds{R}$, $r\geq 0$) \emph{of strain tensors} \cite{BazantJEMT1998,ItskovMRC2004,Itskov2019} (see also \cite{KorobeynikovJElast2019}) is formed using a function $f(\lambda)$ of the form
\begin{equation}\label{2-13}
f_r(\lambda) \equiv
\begin{cases}
\frac{1}{2r} (\lambda^r-\lambda^{-r}), &\text{if}\  r > 0, \\
\log\lambda\ [=\lim \limits_{r \to 0} \frac{1}{2r} (\lambda^r-\lambda^{-r}) ],  &\text{if} \  r=0.
\end{cases}
\end{equation}
\end{description}
Note that strain tensors from the Ba\v{z}ant--Itskov family also belong to the Curnier--Rakotomanana \cite{CurnierET1991} and symmetrically physical (SP) \cite{KorobeynikovJElast2019} families. In addition, the scale function $\log\lambda$ generating the left and right Hencky strain tensors belongs to both families (the Doyle--Ericksen with $n=0$ and the Ba\v{z}ant--Itskov with $r=0$), and each function $f_r(\lambda)$ ($r\neq 0$) is a linear combination of two functions from the Doyle--Ericksen family\footnote{The outstanding properties of the Hencky strain tensor and the corresponding derivation of Hencky's quadratic energy purely based on differential geometric arguments are highlighted in the fundamental studies \cite{NeffJElast2015,NeffARMA2016,XiaoMMMS2005}.}
\begin{equation*}
  f_r(\lambda)= \frac{1}{2}(f^{(r)} + f^{(-r)}).
\end{equation*}
In particular the expression \eqref{2-13} for scale functions with $r=1,\,2$ generates the Pelzer and Mooney \cite{CurnierET1991} strain tensors, respectively (see Table \ref{t1}).

Hereinafter, we assume that all the tensors $\mathbf{H}\in \mathcal{T}^2$ are sufficiently smooth functions of the monotonically increasing parameter $t$ (time), and define the \emph{material time derivative} (\emph{material rate}) of the tensor $\mathbf{H}$: $\dot{\mathbf{H}}\equiv\partial \mathbf{H}/\partial t$. We introduce the \emph{velocity vector} $\mathbf{v}$, the \emph{velocity gradient} $\boldsymbol{\ell}$
\begin{equation}\label{2-15}
\mathbf{v} = \dot{\mathbf{x}},\quad\quad \boldsymbol{\ell} \equiv \text{grad}\mathbf{v}=\mathbf{r}_F\equiv \dot{\mathbf{F}}\cdot \mathbf{F}^{-1},
\end{equation}
the Eulerian \emph{stretching (strain rate) tensor} $\mathbf{d}\in \mathcal{T}^2_{\text{sym}}$, the \emph{vorticity (spin) tensor} $\mathbf{w}\in\mathcal{T}^2_{\text{skew}}$, and the \emph{polar spin tensor} $\boldsymbol{\omega}^R\in\mathcal{T}^2_{\text{skew}}$\footnote{Since we assume that $\mathbf{x}(\mathbf{X},t)\in C^2$ of  $t$, it follows that $\boldsymbol{\ell},\mathbf{d},\mathbf{w},\boldsymbol{\omega}^R \in C^0$ of $t$  (cf., \cite{ScheidlerMM1991}).}
\begin{equation}\label{2-16}
\boldsymbol{\ell} = \mathbf{d} + \mathbf{w},\quad \mathbf{d} \equiv \text{sym}\,\boldsymbol{\ell} = \frac{1}{2}(\boldsymbol{\ell} + \boldsymbol{\ell}^T),\quad \mathbf{w} \equiv \text{skew}\,\boldsymbol{\ell} = \frac{1}{2}(\boldsymbol{\ell} - \boldsymbol{\ell}^T),\quad  \boldsymbol{\omega}^R\equiv \dot{\mathbf{R}}\cdot \mathbf{R}^T.
\end{equation}

Introduce the Lagrangian tensor (cf., \cite{KorobeynikovJElast2008})
\begin{equation*}
    \mathbf{r}_U\equiv\dot{\mathbf{U}}\cdot\mathbf{U}^{-1}=\mathbf{R}^T \cdot (\mathbf{l}-\boldsymbol{\omega}^R)\cdot \mathbf{R},
\end{equation*}
and its symmetric and skew-symmetric parts
\begin{align}\label{2-18}
    \mathbf{r}_U=\widehat{\mathbf{D}}+\mathbf{W},\quad\quad \widehat{\mathbf{D}} &\equiv \text{sym}\,\mathbf{r}_U=\frac{1}{2}(\dot{\mathbf{U}} \cdot \mathbf{U}^{-1}+\mathbf{U}^{-1} \cdot \dot{\mathbf{U}}),\\
    \mathbf{W} &\equiv \text{skew}\,\mathbf{r}_U=\frac{1}{2}(\dot{\mathbf{U}} \cdot \mathbf{U}^{-1}-\mathbf{U}^{-1} \cdot \dot{\mathbf{U}}). \notag
\end{align}
We can show the validity of the relations
\begin{equation}\label{2-19}
    \mathbf{d} = \mathbf{R} \cdot \widehat{\mathbf{D}} \cdot \mathbf{R}^T,\quad\quad \mathbf{w}=\boldsymbol{\omega}^R+\mathbf{R} \cdot \mathbf{W} \cdot \mathbf{R}^T.
\end{equation}
The Lagrangian counterpart $\widehat{\mathbf{D}}$ of the Eulerian tensor $\mathbf{d}$ will be called the \emph{rotated  stretching (strain rate) tensor}, and the tensor $\mathbf{W}$ will be called the \emph{Lagrangian vorticity tensor}.

We have the following expression for the tensor $\mathbf{W}$ (see Eq. $(117)_1$ in \cite{KorobeynikovAM2011}):\footnote{Hereinafter, the notation $\sum_{i\neq j=1}^{m}$ denotes the summation over $i,j=1,\ldots, m$ and $i\neq j$ and this summation is assumed to vanish when $m=1$.}
\begin{equation*}
    \mathbf{W} = \sum_{i\neq j=1}^{m} \frac{\lambda_i-\lambda_j}{\lambda_i+\lambda_j}\mathbf{U}_i\cdot\widehat{\mathbf{D}}\cdot\mathbf{U}_j.
\end{equation*}
Substituting this expression for $\mathbf{W}$ into $\eqref{2-19}_2$, we obtain the relation between the spin tensors $\mathbf{w}$ and $\boldsymbol{\omega}^R$ (see also Eq. $(117)_2$ in \cite{KorobeynikovAM2011}):
\begin{equation*}
    \mathbf{w}=\boldsymbol{\omega}^R + \sum_{i\neq j=1}^{m} \frac{\lambda_i-\lambda_j}{\lambda_i+\lambda_j}\mathbf{V}_i\cdot\mathbf{d}\cdot\mathbf{V}_j.
\end{equation*}

Let $\mathbf{h}\in \mathcal{T}^2$ be an Eulerian tensor. We define objective Eulerian tensors that are the \emph{upper and lower Oldroyd rates} of the tensor $\mathbf{h}$ \cite{OldroydPRSLA1950,OldroydPRSLA1958} (see also \cite{AubramMMS2025,FedericoMMS2025,FialaIJNLM2016,Fiala2017,FialaZAMP2020}):
\begin{equation*}
   \left(\frac{\text{D}^{\sharp}}{\text{D}t}[\mathbf{h}] = \right) \mathbf{h}^{\sharp} \equiv \dot{\mathbf{h}}-\mathbf{h}\cdot\mathbf{l}^T - \mathbf{l}\cdot \mathbf{h},\quad\quad
   \left(\frac{\text{D}^{\flat}}{\text{D}t}[\mathbf{h}] = \right) \mathbf{h}^{\flat} \equiv \dot{\mathbf{h}}+ \mathbf{h}\cdot\mathbf{l} + \mathbf{l}^T\cdot \mathbf{h}.
\end{equation*}

We introduce \emph{the family of Eulerian corotational rates} of the Eulerian tensor $\mathbf{h}$ as tensors of the form (see, e.g., \cite{AubramMMS2025,FedericoMMS2025,FialaIJNLM2016,Fiala2017,FialaZAMP2020})
\begin{equation}\label{2-23}
    \left(\frac{\text{D}^{\circ}}{\text{D}t}[\mathbf{h}] = \right) \mathbf{h}^{\omega}\equiv \dot{\mathbf{h}} + \mathbf{h}\cdot\boldsymbol{\omega} - \boldsymbol{\omega}\cdot \mathbf{h}
\end{equation}
associated with \emph{the spin tensors} $\boldsymbol{\omega}\in \mathcal{T}^2_{\text{skew}}$ which belong to \emph{the family of continuous material spin tensors} \cite{XiaoIJSS1998,XiaoJElast1998} (see also \cite{KorobeynikovAM2011}).

Some objective Eulerian rates of Eulerian strain tensors from the Hill family given in Table \ref{t1} are related to the stretching tensor $\mathbf{d}$  by the following equalities (see, e.g., \cite{BruhnsPRSLA2004,KorobeynikovAAM2020,KorobeynikovJElast2021}):\footnote{Expressions for the spin tensor $\boldsymbol{\omega}^{\text{log}}$ associated with the logarithmic corotational tensor rate $\mathbf{h}^{\text{log}}$ of any Eulerian tensor $\mathbf{h}$ will be given below.}
\begin{equation*}
    \mathbf{e}^{(2)\,\sharp}=\mathbf{d},\quad\quad\quad \mathbf{e}^{(-2)\,\flat}=\mathbf{d},\quad\quad\quad \mathbf{e}^{(0)\,\text{log}}=\mathbf{d}.
\end{equation*}

Let $\mathbf{H}\in \mathcal{T}^2$ be a Lagrangian tensor. We introduce the \emph{upper and lower Lagrangian Oldroyd's rates} of the tensor $\mathbf{H}$ (see, e.g., \cite{KorobeynikovJElast2008,KorobeynikovJElast2021}):
\begin{equation*}
    \mathbf{H}^{\sharp} \equiv \dot{\mathbf{H}}-\mathbf{H}\cdot\mathbf{r}_U^T - \mathbf{r}_U\cdot \mathbf{H},\quad\quad\quad
    \mathbf{H}^{\flat} \equiv \dot{\mathbf{H}}+ \mathbf{H}\cdot\mathbf{r}_U + \mathbf{r}_U^T\cdot \mathbf{H}.
\end{equation*}

We also introduce the \emph{family of Lagrangian corotational rates} of the Lagrangian tensor $\mathbf{H}$ in the form of tensors
\begin{equation}\label{2-26}
    \mathbf{H}^{\Omega}\equiv \dot{\mathbf{H}} + \mathbf{H}\cdot\boldsymbol{\Omega} - \boldsymbol{\Omega}\cdot \mathbf{H},
\end{equation}
associated with \emph{spin tensors} $\boldsymbol{\Omega}\in \mathcal{T}^2_{\text{skew}}$ which belong to the \emph{family of continuous material spin tensors} \cite{KorobeynikovAM2011}.

Some oblective Lagrangian rates of Lagrangian strain tensors from the Hill family given in Table \ref{t1} are related to the Lagrangian stretching tensor $\widehat{\mathbf{D}}$ as follows (see, e.g., \cite{KorobeynikovJElast2021}):\footnote{Expressions for the spin tensor $\boldsymbol{\Omega}^{\text{log}}$ associated with the logarithmic corotational tensor rate $\mathbf{H}^{\text{log}}$ of any Lagrangian tensor $\mathbf{H}$ will be given below.}
\begin{equation*}
    \mathbf{E}^{(2)\,\sharp}=\widehat{\mathbf{D}},\quad\quad\quad \mathbf{E}^{(-2)\,\flat}=\widehat{\mathbf{D}},\quad\quad\quad \mathbf{E}^{(0)\,\log}=\widehat{\mathbf{D}}.
\end{equation*}

\begin{definition}
\label{def:2-2}
Let $\mathbf{h},\,\mathbf{H}\in \mathcal{T}^2$ be Eulerian and Lagrangian tensors, respectively. We will call these tensors \emph{objective counterparts} of each other if they are related by the equalities
\begin{equation}\label{2-28}
    \mathbf{h}=\mathbf{R} \cdot \mathbf{H} \cdot \mathbf{R}^T \quad \Leftrightarrow \quad \mathbf{H}=\mathbf{R}^T \cdot \mathbf{h} \cdot \mathbf{R}.
\end{equation}
\end{definition}

\begin{proposition}
\label{Pr:2-1}
Let $\mathbf{h},\,\mathbf{H}\in \mathcal{T}^2$ be Eulerian and Lagrangian tensors, respectively, which are objective counterparts of each other. Then the Oldroyd tensor rates $\mathbf{h}^{\sharp}$, $\mathbf{h}^{\flat}$ and $\mathbf{H}^{\sharp}$, $\mathbf{H}^{\flat}$ are objective counterparts of each other, i.e., the following equalities hold:
\begin{equation*}
    \mathbf{h}^{\sharp}=\mathbf{R} \cdot \mathbf{H}^{\sharp} \cdot \mathbf{R}^T,\quad \mathbf{h}^{\flat}=\mathbf{R} \cdot \mathbf{H}^{\flat} \cdot \mathbf{R}^T, \quad \Leftrightarrow \quad \mathbf{H}^{\sharp}=\mathbf{R}^T \cdot \mathbf{h}^{\sharp} \cdot \mathbf{R},\quad \mathbf{H}^{\flat}=\mathbf{R}^T \cdot \mathbf{h}^{\flat} \cdot \mathbf{R}.
\end{equation*}
\end{proposition}

\begin{proposition}
\label{Pr:2-2}
Let $\mathbf{h},\,\mathbf{H}\in \mathcal{T}^2$ be Eulerian and Lagrangian tensors, respectively, which are objective counterparts of each other. Then the necessary and sufficient condition for the corotational rates $\mathbf{h}^{\omega}$ and $\mathbf{H}^{\Omega}$ of the tensors $\mathbf{h}$ and $\mathbf{H}$ to be objective counterparts of each other is given by
\begin{equation}\label{2-30}
  \boldsymbol{\omega}=\boldsymbol{\omega}^R + \mathbf{R} \cdot \boldsymbol{\Omega} \cdot \mathbf{R}^T \quad \Leftrightarrow \quad \boldsymbol{\Omega} = \mathbf{R}^T \cdot (\boldsymbol{\omega} - \boldsymbol{\omega}^R) \cdot \mathbf{R},
\end{equation}
where $\boldsymbol{\omega},\,\boldsymbol{\Omega}\in \mathcal{T}^2_{\text{skew}}$ are spin tensors associated with the corotational tensor rates $\mathbf{h}^{\omega}$ and $\mathbf{H}^{\Omega}$, respectively.\footnote{Proposition \ref{Pr:2-2} was formulated and proved in \cite{Korobeynikov2023} (see Proposition 2.1 in \cite{Korobeynikov2023}).}\\
\end{proposition}

\begin{proposition}
\label{Pr:2-2a}
Let $\mathbf{a},\,\mathbf{b}\in \mathcal{T}^2$ be Eulerian and $\mathbf{A},\,\mathbf{B}\in \mathcal{T}^2$ be Lagrangian tensors, respectively, which are objective counterparts of each other, i.e., (see Definition \ref{def:2-2})
\begin{equation*}
    \mathbf{a}=\mathbf{R} \cdot \mathbf{A} \cdot \mathbf{R}^T,\quad \mathbf{b}=\mathbf{R} \cdot \mathbf{B} \cdot \mathbf{R}^T \quad \Leftrightarrow \quad \mathbf{A}=\mathbf{R}^T \cdot \mathbf{a} \cdot \mathbf{R},\quad \mathbf{B}=\mathbf{R}^T \cdot \mathbf{b} \cdot \mathbf{R}.
\end{equation*}
Then the following equality holds:
\begin{equation*}
  \mathbf{A}:\dot{\mathbf{B}} = \mathbf{a}:\mathbf{b}^{\mathrm{GN}}\quad \ (\Leftrightarrow \quad \langle A,\frac{\mathrm{D}}{\mathrm{D}t}[B]\rangle =
  \langle a,\frac{\mathrm{D}^{\mathrm{GN}}}{\mathrm{D}t}[b]\rangle).
\end{equation*}
\end{proposition}

Next, we explore the spin tensors $\boldsymbol{\Omega}$ and $\boldsymbol{\omega}$ from the family of continuous material spin tensors in two alternative versions (see, e.g., \cite{KorobeynikovAM2011,KorobeynikovAAM2020,NeffAM2025}):
\begin{description}
  \item[(i)] \emph{r-version}
 \begin{equation}\label{2-31}
    \boldsymbol{\Omega}=\boldsymbol{\Psi}_r(\mathbf{U},\widehat{\mathbf{D}}),\quad\quad\quad \boldsymbol{\omega}= \boldsymbol{\omega}^R + \boldsymbol{\Psi}_r(\mathbf{V},\mathbf{d}).
\end{equation}
Here $\boldsymbol{\Psi}_r(\mathbf{U},\widehat{\mathbf{D}})\in \mathcal{T}^2_{\text{skew}}$ is an isotropic tensor function of the arguments $\mathbf{U}$ and $\widehat{\mathbf{D}}$ which is linear in $\widehat{\mathbf{D}}$,\footnote{The isotropic tensor functions $\boldsymbol{\Psi}_r(\mathbf{V},\mathbf{d})$ and $\boldsymbol{\Psi}_g(\mathbf{V},\mathbf{d})$ are derived from the isotropic tensor functions $\boldsymbol{\Psi}_r(\mathbf{U},\widehat{\mathbf{D}})$ and $\boldsymbol{\Psi}_g(\mathbf{U},\widehat{\mathbf{D}})$ by replacing the Lagrangian tensor arguments $\mathbf{U}$ and $\widehat{\mathbf{D}}$ with their Eulerian counterparts $\mathbf{V}$ and $\mathbf{d}$, respectively.}
\begin{equation}\label{2-32}
    \boldsymbol{\Psi}_r(\mathbf{U},\widehat{\mathbf{D}}) \equiv \sum_{i\neq j=1}^{m}r_{ij} \mathbf{U}_i \cdot \widehat{\mathbf{D}} \cdot \mathbf{U}_j,
\end{equation}
where we used the notation
\begin{equation*}
    r_{ij} \equiv r(\lambda_i,\lambda_j),\quad\quad r_{ji}=-r_{ij}.
\end{equation*}
The properties of the function $r(x,y)$ are discussed in detail in \cite{XiaoIJSS1998,XiaoJElast1998} (see also \cite{KorobeynikovAM2011}).
  \item[(ii)] \emph{g-version}
\begin{equation}\label{2-34}
    \boldsymbol{\Omega}=\mathbf{W} + \boldsymbol{\Psi}_g(\mathbf{U},\widehat{\mathbf{D}}),\quad\quad\quad \boldsymbol{\omega}= \mathbf{w} +   \boldsymbol{\Psi}_g(\mathbf{V},\mathbf{d}).
\end{equation}
Here $\boldsymbol{\Psi}_g(\mathbf{U},\widehat{\mathbf{D}})\in \mathcal{T}^2_{\text{skew}}$ is an isotropic tensor function of the tensors $\mathbf{U}$ and $\widehat{\mathbf{D}}$ which is linear in $\widehat{\mathbf{D}}$,
\begin{equation}\label{2-35}
    \boldsymbol{\Psi}_g(\mathbf{U},\widehat{\mathbf{D}}) \equiv \sum_{i\neq j=1}^{m}g_{ij} \mathbf{U}_i \cdot \widehat{\mathbf{D}} \cdot \mathbf{U}_j,
\end{equation}
where we used the notation
\begin{equation*}
    g_{ij} \equiv g(\lambda_i,\lambda_j)\equiv r(\lambda_i,\lambda_j)-\frac{\lambda_i - \lambda_j}{\lambda_i + \lambda_j},\quad\quad\quad g_{ji}=-g_{ij}.
\end{equation*}
\end{description}

Below we give expressions for the quantities $g_{ij}$ and $r_{ij}$ generating some spin tensors from the family of continuous material spin tensors and the classical corotational symmetric tensor rates associated with them $(i,j=1,\ldots,m)$:
\begin{itemize}
  \item The \emph{Zaremba--Jaumann tensor rates} $\mathbf{H}^{\text{ZJ}}$ and $\left(\frac{\text{D}^{\text{ZJ}}}{\text{D}t}[\mathbf{h}] = \right)\, \mathbf{h}^{\text{ZJ}}$ associated with the spin tensors $\mathbf{W}$ and $\mathbf{w}$
\begin{equation}\label{2-37}
  g_{ij}^{\text{ZJ}}=0;\quad\quad r_{ij}^{\text{ZJ}}\equiv\frac{\lambda_i - \lambda_j}{\lambda_i + \lambda_j}.
\end{equation}
  \item The material rate $\dot{\mathbf{\mathbf{H}}}$ and the \emph{Green--Naghdi tensor rate} $\left(\frac{\text{D}^{\text{GN}}}{\text{D}t}[\mathbf{h}] = \right)\, \mathbf{h}^{\text{GN}}$ associated with the spin tensors $(\boldsymbol{\Omega}=)\mathbf{0}$ and $\boldsymbol{\omega}^R$
\begin{equation}\label{2-38}
  g_{ij}^{\text{GN}}=\frac{\lambda_j - \lambda_i}{\lambda_i + \lambda_j};\quad\quad r_{ij}^{\text{GN}}=0.
\end{equation}
  \item The \emph{Gurtin--Spear tensor rates} $\mathbf{H}^{\text{GS}}$ and $\left(\frac{\text{D}^{\text{GS}}}{\text{D}t}[\mathbf{h}] = \right)\,\mathbf{h}^{\text{GS}}$ \cite{GurtinIJSS1983} associated with the twirl tensors $\boldsymbol{\Omega}^{\text{GS}}$ and $\boldsymbol{\omega}^{\text{GS}}$ of Lagrangian and Eulerian triads, respectively
\begin{equation}\label{2-39}
    g_{ij}^{\text{GS}}\equiv \frac{\lambda_i^2 + \lambda_j^2}{\lambda_j^2 - \lambda_i^2};
   \quad\quad r_{ij}^{\text{GS}}\equiv \frac{2\lambda_i \lambda_j}{\lambda_j^2 - \lambda_i^2}.
\end{equation}
  \item The \emph{logarithmic tensor rates} $\mathbf{H}^{\log}$ and $\left(\frac{\text{D}^{\log}}{\text{D}t}[\mathbf{h}] = \right) \mathbf{h}^{\log}$ \cite{XiaoAM1997,Zhilin2013} associated with the spin tensor $\boldsymbol{\Omega}^{\log}$ and $\boldsymbol{\omega}^{\log}$
\begin{equation}\label{2-40}
    g_{ij}^{\log}\equiv \frac{\lambda_i^2 + \lambda_j^2}{\lambda_j^2 - \lambda_i^2}+\frac{1}{\log\lambda_i-\log\lambda_j};
\quad\quad r_{ij}^{\log}\equiv \frac{2\lambda_i \lambda_j}{\lambda_j^2 - \lambda_i^2} + \frac{1}{\log \lambda_i - \log \lambda_j}.
\end{equation}
\end{itemize}

\begin{remark}
\label{rem:2-1}
The twirl tensors $\boldsymbol{\Omega}^{\text{GS}}$ and $\boldsymbol{\omega}^{\text{GS}}$ do not belong to the family of continuous material spin tensors, since the scalar functions $g_{ij}^{\text{GS}}$ and $r_{ij}^{\text{GS}}$ generating these tensors have a discontinuity of the second kind at the point of coincidence of the principal stretches (i.e., for example, for $\lambda_3\rightarrow \lambda_2$) even if the law of motion $\mathbf{x}(\mathbf{X},t)$ is a sufficiently smooth function of $t$ (see, e.g., \cite{BruhnsAM2002,KorobeynikovAM2011,ScheidlerMM1991,XiaoArchMech1998,XiaoIJSS1998,XiaoJElast1998}). Moreover, when the principal stretches coalesce, the representation of the twirl tensors $\boldsymbol{\Omega}^{\text{GS}}$ and $\boldsymbol{\omega}^{\text{GS}}$ lose uniqueness (see, e.g., \cite{KorobeynikovAM2011}). However, for the laws of motion $\mathbf{x}(\mathbf{X},t)$ considered in this paper, the inequalities $\lambda_1\neq \lambda_2\neq \lambda_3\neq \lambda_1$ hold for $t> t_0$, i.e., the eigenindex $m=3$ and the coalescence of the principal stretches is ruled out. At the same time, the remaining requirements for the twirl tensors $\boldsymbol{\Omega}^{\text{GS}}$ and $\boldsymbol{\omega}^{\text{GS}}$ to belong to the family of continuous material spin tensors (see \cite{XiaoArchMech1998,XiaoIJSS1998,XiaoJElast1998} and \cite{KorobeynikovAM2011} for details) are satisfied. This allowed us to include the twirl tensors $\boldsymbol{\Omega}^{\text{GS}}$ and $\boldsymbol{\omega}^{\text{GS}}$ in the family of continuous material spin tensors for the laws of motion considered.
\end{remark}
\vspace*{2mm}

\begin{remark}
\label{rem:2-2}
Let the tensors $\mathbf{h}$ and $\mathbf{H}$ be Eulerian and Lagrangian counterparts of each other.  Xiao et al. \cite{XiaoArchMech1998,XiaoIJSS1998,XiaoJElast1998} generated a spin tensor of the form (see Eq. (51) in \cite{XiaoIJSS1998} and Eq. (4.42) in \cite{XiaoJElast1998})
\begin{equation}\label{2-41}
    \widetilde{\boldsymbol{\omega}}^L = \mathbf{w} + \widetilde{\boldsymbol{\Psi}}(\mathbf{V},\mathbf{d}),\quad\quad \widetilde{\boldsymbol{\Psi}}(\mathbf{V},\mathbf{d}) \equiv \sum_{i\neq j=1}^{m}r_{ij}^{\text{GS}} \mathbf{V}_i \cdot \mathbf{d} \cdot \mathbf{V}_j
\end{equation}
associated with the Eulerian corotational rate of an Eulerian tensor $\mathbf{h}$, which, according to \cite{XiaoArchMech1998,XiaoIJSS1998,XiaoJElast1998}, is the Eulerian counterpart of the Lagrangian corotational rate of the Lagrangian tensor $\mathbf{H}$ associated with the twirl tensor of the Lagrangian triad $\boldsymbol{\Omega}^{\text{GS}}$. The error in this interpretation of the tensor $\widetilde{\boldsymbol{\omega}}^L$ arises from the use of relations between spin tensors of the form \eqref{2-30}, but with the replacement of the polar spin tensor $\boldsymbol{\omega}^R$ by the vorticity tensor $\mathbf{w}$ (see Eq. (4.45) in \cite{XiaoJElast1998}). Replacing the vorticity tensor $\mathbf{w}$ in Eq. (4.45) in \cite{XiaoJElast1998} by the polar spin tensor $\boldsymbol{\omega}^R$, as required by the relation between the tensors  $\boldsymbol{\omega}$ and $\boldsymbol{\Omega}$ in \eqref{2-30}, we obtain, instead of the mechanically  meaningless tensor $\widetilde{\boldsymbol{\omega}}^L$, the expected twirl tensor of the Eulerian triad $\boldsymbol{\omega}^{\text{GS}}$ on the left-hand side of $\eqref{2-30}_1$. This misinterpretation of the spin tensor $\widetilde{\boldsymbol{\omega}}^L$ \cite{XiaoArchMech1998,XiaoIJSS1998,XiaoJElast1998} have been referenced and even used by other authors in testing hypoelastic models based on corotational stress rates in homogeneous deformation problems (see, e.g., \cite{LinIJNME2002,LinEJMAS2003,LinZAMM2003,LinZAMP2024,ZhouFEAD2003}). Moreover, when using the above-mentioned hypoelastic model with the corotational stress rate associated with the tensor $\widetilde{\boldsymbol{\omega}}^L$, Lin \cite{LinZAMP2024} even omitted the vorticity tensor $\mathbf{w}$ on the right-hand side of $\eqref{2-41}_1$ (see Eq. $(15)_3$ in \cite{LinZAMP2024}), which leads to the non-objectivity of the stress rate based on this expression for the tensor $\widetilde{\boldsymbol{\omega}}^L$.
\end{remark}

\subsection{Hooke-like isotropic hyper-/hypo-elastic material models}
\label{sec:2-3}

Consider the family of \emph{Hill's linear isotropic hyperelastic} (HLIH) material models with constitutive relations of the form \cite{Hill1979} (see also \cite{KorobeynikovJElast2019})
\begin{equation*}
  \mathbf{S} = 2\mu\, \mathbf{E} + \lambda\, (\text{tr}\,\mathbf{E})\,\mathbf{I} \quad \Leftrightarrow \quad \mathbf{s} =  2\mu\, \mathbf{e} + \lambda\, (\text{tr}\,\mathbf{e})\,\mathbf{I}
\end{equation*}
in the Lagrangian and Eulerian versions, respectively. Hereinafter, $\lambda$ and $\mu$ are the \emph{Lam\'{e} parameters}, $\textbf{E}$ and $\textbf{e}$ are any Lagrangian and Eulerian strain tensors from the Hill family, and $\textbf{S}$ and $\textbf{s}$ are work-conjugate stress tensors, so that the following equalities hold \cite{Hill1979} (see also \cite{KorobeynikovAM2018,XiaoArchMech1998} and Proposition \ref{Pr:2-2a}):
\begin{equation*}
  w = \mathbf{S}:\dot{\mathbf{E}} \quad \Leftrightarrow \quad w = \mathbf{s}:\mathbf{e}^{\text{GN}}.
\end{equation*}
Here $w\equiv \boldsymbol{\tau}:\mathbf{d} = \bar{\boldsymbol{\tau}}:\widehat{\mathbf{D}}$ is the \emph{stress power} (objective scalar) per unit volume of the deformable body $\mathfrak{B}$ in the reference configuration; $\boldsymbol{\tau}$ and $\bar{\boldsymbol{\tau}}$ are the symmetric Eulerian \emph{Kirchhoff} and Lagrangian \emph{rotated Kirchhoff stress tensors}
\begin{equation*}
\boldsymbol{\tau} \equiv J\boldsymbol{\sigma}, \quad\quad \bar{\boldsymbol{\tau}} \equiv \mathbf{R}^T \cdot \boldsymbol{\tau} \cdot \mathbf{R}, 
\end{equation*}
and $\boldsymbol{\sigma}$ denotes the symmetric Eulerian \emph{Cauchy (true) stress tensor}.

In \cite{KorobeynikovJElast2019}, the Lagrangian $\bar{\boldsymbol{\tau}}$ and Eulerian $\boldsymbol{\tau}$ stress tensors for these materials were expressed as
\begin{equation}\label{2-42}
    \bar{\boldsymbol{\tau}}=\sum^m_{i=1}g_i f^\prime(\lambda_i)\lambda_i \mathbf{U}_i\quad \Leftrightarrow \quad \boldsymbol{\tau}=\sum^m_{i=1}g_i f^\prime(\lambda_i)\lambda_i \mathbf{V}_i,
\end{equation}
where $f(\lambda)$ is any scale function that generates strain tensors $\mathbf{E}$ and $\mathbf{e}$ from the Hill family,
\begin{equation}\label{2-43}
    g_i\equiv \lambda (f_1+f_2+f_3)+2\mu f_i, \quad f_i\equiv f(\lambda_i),\quad f^\prime(\lambda_i)\equiv \frac{\partial f(\lambda_i)}{\partial \lambda_i}\quad (i=1,\ldots m).
\end{equation}

The constitutive relations for some models from the HLIH family are given in Table \ref{t2}.
\begin{table}
\caption{Hill's linear isotropic hyperelastic material models}
\label{t2}
\scriptsize
\begin{tabular}{l|lll}
\hline\noalign{\smallskip}
 Model's   & Strain &  \multicolumn{2}{c}{Constitutive relations}                  \\
 name      & tensor & Lagrangian version           &       Eulerian version        \\
\noalign{\smallskip}\hline\noalign{\smallskip}
   HLIH-H   & Hencky & $\bar{\boldsymbol{\tau}}\ \ \ = 2\mu\,\mathbf{E}^{(0)} + \lambda\,(\text{tr}\,\mathbf{E}^{(0)})\,\mathbf{I}$ & $\boldsymbol{\tau}\ \ \,= 2\mu\,\mathbf{e}^{(0)} + \lambda\,(\text{tr}\,\mathbf{e}^{(0)})\,\mathbf{I}$\\
   HLIH-P   & Pelzer & $\mathbf{S}^P\ = 2\mu\,\mathbf{E}^P + \lambda\,(\text{tr}\,\mathbf{E}^P)\,\mathbf{I}$ & $\mathbf{s}^P\ = 2\mu\,\mathbf{e}^P + \lambda\,(\text{tr}\,\mathbf{e}^P)\,\mathbf{I}$ \\
   HLIH-M   & Mooney & $\mathbf{S}^M\,  = 2\mu\,\mathbf{E}^M + \lambda\,(\text{tr}\,\mathbf{E}^M)\,\mathbf{I}$ & $\mathbf{s}^M\, = 2\mu\,\mathbf{e}^M + \lambda\,(\text{tr}\,\mathbf{e}^M)\,\mathbf{I}$ \\
\noalign{\smallskip}\hline
\end{tabular}
\end{table}
This table introduces the Pelzer and Mooney Lagrangian $\mathbf{S}^P$ and $\mathbf{S}^M$ and Eulerian $\mathbf{s}^P$ and $\mathbf{s}^M$ stress tensors. General basis-free expressions for these tensors are presented in \cite{KorobeynikovJElast2019}. For the hyperelastic material models considered here, these general expressions are simplified because for these models all Lagrangian (or Eulerian) stress tensor work-conjugate to Lagrangian (or Eulerian) strain tensors are coaxial with the right (or left) stretch tensor $\mathbf{U}$ (or $\mathbf{V}$). The simplified expressions for the stress tensors considered have the following form ($m$ is the eigenindex of the $\mathbf{U}$ and $\mathbf{V}$ tensors):
\begin{align*}
  \mathbf{S}^P \equiv  \sum^m_{i=1}\frac{2}{\lambda_i+\lambda_i^{-1}}\mathbf{U}_i\cdot \bar{\boldsymbol{\tau}}\cdot \mathbf{U}_i\quad \Leftrightarrow\quad
  &  \mathbf{s}^P \equiv  \sum^m_{i=1}\frac{2}{\lambda_i+\lambda_i^{-1}}\mathbf{V}_i\cdot \boldsymbol{\tau}\cdot \mathbf{V}_i, \\
  \mathbf{S}^M \equiv  \sum^m_{i=1}\frac{2}{\lambda_i^2+\lambda_i^{-2}}\mathbf{U}_i\cdot \bar{\boldsymbol{\tau}}\cdot \mathbf{U}_i\quad \Leftrightarrow\quad
  &  \mathbf{s}^M \equiv  \sum^m_{i=1}\frac{2}{\lambda_i^2+\lambda_i^{-2}}\mathbf{V}_i\cdot \boldsymbol{\tau}\cdot \mathbf{V}_i. \notag
\end{align*}
In particular, for the HLIH-H (or simply the Hencky) material model, the constitutive relations can be written in a simple and elegant form as
\begin{align}\label{2-44}
    {} & \bar{\boldsymbol{\tau}} = 2\mu\, \mathbf{E}^{(0)} + \lambda\, (\log J)\,\mathbf{I} = 2\mu\,\log \mathbf{U} + \lambda\,\text{tr}(\log \mathbf{U})\,\mathbf{I} \\ \Leftrightarrow \quad & \boldsymbol{\tau} = 2\mu\, \mathbf{e}^{(0)} + \lambda\, (\log J)\,\mathbf{I} = 2\mu\,\log \mathbf{V} + \lambda\,\text{tr}(\log \mathbf{V})\,\mathbf{I}. \notag
\end{align}

We now introduce another family of Hooke-like Ogden-type isotropic hyperelastic material models using the basis-free expressions for constitutive relations for isotropic hyperelastic incompressible materials \cite{KorobeynikovAAM2025} derived by generalizing constitutive relations for the well-known Ogden's model \cite{OgdenPRSLA1972a}. However, for greater generality, we consider a mixed version (cf., \cite{Korobeynikov2026}) of constitutive relations for this model taking into account  material compressibility \cite{KorobeynikovAAM2023} whose prototype is Ogden's constitutive relations for compressible isotropic hyperelastic material models \cite{OgdenPRSLA1972b} (see also \cite{Wriggers2008}). Next, we confine ourselves to two typical one-power models and then  consider a family of two-power Ogden-type models of special form.

The constitutive relations for the one-power Ogden's models considered in this paper are presented in Table \ref{t3}.
\begin{table}
\caption{One-power Ogden's compressible isotropic hyperelastic material models}
\label{t3}
\scriptsize
\begin{tabular}{l|lll}
\hline\noalign{\smallskip}
 Model's   & Strain &  \multicolumn{2}{c}{Constitutive relations}                  \\
 name      & tensors & Lagrangian version                                                               & Eulerian version        \\
\noalign{\smallskip}\hline\noalign{\smallskip}
  Ogden-A   & $\mathbf{E}^{(2)}$, $\mathbf{e}^{(2)}$   & $\bar{\boldsymbol{\tau}} = 2\mu\, \mathbf{E}^{(2)} + \lambda\, (\log J)\,\mathbf{I}$ & $\boldsymbol{\tau} = 2\mu\, \mathbf{e}^{(2)} + \lambda\, (\log J)\,\mathbf{I}$ \\
  Ogden-B   & $\mathbf{E}^{(-2)}$, $\mathbf{e}^{(-2)}$ & $\bar{\boldsymbol{\tau}} = 2\mu\, \mathbf{E}^{(-2)} + \lambda\, (\log J)\,\mathbf{I}$ & $\boldsymbol{\tau} = 2\mu\, \mathbf{e}^{(-2)} + \lambda\, (\log J)\,\mathbf{I}$ \\
\noalign{\smallskip}\hline
\end{tabular}
\end{table}
Both of these models are special cases of the mixed-type compressible isotropic hyperelastic Mooney--Rivlin material model (see Appendix \ref{sec:A}). In addition, the Ogden-A material model is known as the mixed-type compressible isotropic hyperelastic neo-Hookean model \cite{Korobeynikov2026} or the Simo--Pister one \cite{SimoCMAME1984}.

We generate a family of two-power Ogden-type models with the following constitutive relations for compressible isotropic hyperelastic materials:
\begin{equation}\label{2-47}
    \bar{\boldsymbol{\tau}} = 2\mu\, \mathbf{E}_r + \lambda\, (\log J)\,\mathbf{I} \quad \Leftrightarrow \quad \boldsymbol{\tau} = 2\mu\, \mathbf{e}_r + \lambda\, (\log J)\,\mathbf{I}.
\end{equation}
Here $\mathbf{E}_r$ and $\mathbf{e}_r$ are strain tensors from the Ba\v{z}ant--Itskov family generated by scale functions of form \eqref{2-13}. For $r=0$, this model is the Hencky isotropic hyperelastic material model (or the HLIH-H model) with constitutive relations \eqref{2-44}.\footnote{Note that the Hencky model is included as a special case in the generalized Ogden model \cite{KorobeynikovAAM2023,KorobeynikovAAM2025}.} For $r>0$, using expressions \eqref{2-13}, we can rewrite constitutive relations \eqref{2-47} as
\begin{equation}\label{2-47a}
    \bar{\boldsymbol{\tau}} = \mu\, \mathbf{E}^{(r)} + \mu\, \mathbf{E}^{(-r)} + \lambda\, (\log J)\,\mathbf{I} \quad \Leftrightarrow \quad \boldsymbol{\tau} = \mu\, \mathbf{e}^{(r)} + \mu\, \mathbf{e}^{(-r)} + \lambda\, (\log J)\,\mathbf{I}.
\end{equation}
The fact that constitutive relations \eqref{2-47} can be rewritten in the form \eqref{2-47a} justifies the term ``the family of two-power Ogden-type models.'' The constitutive relations for the two-power Ogden's models considered in this paper are exposed in Table \ref{t4}.
\begin{table}
\caption{Ogden's type Ba\v{z}ant--Itskov strain tensor based isotropic hyperelastic material models}
\label{t4}
\scriptsize
\begin{tabular}{l|lll}
\hline\noalign{\smallskip}
 Model's   & Strain &  \multicolumn{2}{c}{Constitutive relations}                  \\
 name      & tensor & Lagrangian version                                                               & Eulerian version        \\
\noalign{\smallskip}\hline\noalign{\smallskip}
   OBI-H   & Hencky & $\bar{\boldsymbol{\tau}}=2\mu\,\log\,\mathbf{U} + \lambda\,(\log\,J)\,\mathbf{I}$ & $\boldsymbol{\tau} = 2\mu\,\log\,\mathbf{V} + \lambda\,(\log\,J)\,\mathbf{I}$ \\
   OBI-P   & Pelzer & $\bar{\boldsymbol{\tau}}= \mu\,(\mathbf{U} + \mathbf{U}^{-1}) + \lambda\,(\log\,J)\,\mathbf{I}$ & $\boldsymbol{\tau} = \mu\,(\mathbf{V} + \mathbf{V}^{-1}) +  \lambda\,(\log\,J)\,\mathbf{I}$ \\
   OBI-M   & Mooney & $\bar{\boldsymbol{\tau}}=\frac{\mu}{2}(\mathbf{C} + \mathbf{C}^{-1}) + \lambda\,(\log\,J)\,\mathbf{I}$ & $\boldsymbol{\tau} = \frac{\mu}{2}(\mathbf{c} + \mathbf{c}^{-1}) + \lambda\,(\log\,J)\,\mathbf{I}$ \\
\noalign{\smallskip}\hline
\end{tabular}
\end{table}
Note that the OBI-M model is a mixed-type compressible (isotropic hyperelastic) Mooney--Rivlin material model with the parameter values $\mu_1=\mu_2=\mu/2$ (see Appendix \ref{sec:A}).

In addition to testing the Hooke-like isotropic hyperelastic material models considered above, we will also test Hooke-like isotropic hypo-elastic models (see, e.g., \cite{AubramMMS2025,KorobeynikovAAM2020,KorobeynikovJElast2021,Korobeynikov2023,KorobeynikovZAMM2024,NeffJMPS2025,XiaoAM1997,XiaoJElast1997,XiaoAM1999,XiaoJElast1999}). The first two models use the upper and lower Oldroyd stress rates as objective stress rates and have the following constitutive relations:
\begin{itemize}
  \item for the upper Oldroyd stress rates,
\begin{equation}\label{2-47b}
    \bar{\boldsymbol{\tau}}^{\sharp} = 2\mu\, \widehat{\mathbf{D}} + \lambda\, (\text{tr}\,\widehat{\mathbf{D}})\,\mathbf{I}  \quad \Leftrightarrow \quad  \boldsymbol{\tau}^{\sharp} = 2\mu\, \mathbf{d} + \lambda\, (\text{tr}\,\mathbf{d})\,\mathbf{I};
\end{equation}
  \item for the lower Oldroyd stress rates,
\begin{equation}\label{2-48}
    \bar{\boldsymbol{\tau}}^{\flat} = 2\mu\, \widehat{\mathbf{D}} + \lambda\, (\text{tr}\,\widehat{\mathbf{D}})\,\mathbf{I}  \quad \Leftrightarrow \quad \boldsymbol{\tau}^{\flat} = 2\mu\, \mathbf{d} + \lambda\, (\text{tr}\,\mathbf{d})\,\mathbf{I}.
\end{equation}
\end{itemize}

For the other constitutive relations for Hooke-like isotropic hypoelasticity, as objective stress rates we use corotational stress rates associated with spin tensors from the family of material spin tensors\footnote{Note that the Hencky model is also included in the rate formulation $\eqref{2-49}_2$ by choosing \linebreak  $\frac{\text{D}^{\circ}}{\text{D}t}[\boldsymbol{\tau}] = \frac{\text{D}^{\log}}{\text{D}t}[\boldsymbol{\tau}] = 2\mu\, \mathbf{d} + \lambda\, (\text{tr}\,\mathbf{d})\,\mathbf{I}$ \cite{XiaoJElast1999,XiaoAM1999}.}
\begin{equation}\label{2-49}
    \bar{\boldsymbol{\tau}}^{\Omega} = 2\mu\, \widehat{\mathbf{D}} + \lambda\, (\text{tr}\,\widehat{\mathbf{D}})\,\mathbf{I}  \quad \Leftrightarrow \quad \left(\frac{\text{D}^{\circ}}{\text{D}t}[\boldsymbol{\tau}] = \right)\,\boldsymbol{\tau}^{\omega} = 2\mu\, \mathbf{d} + \lambda\, (\text{tr}\,\mathbf{d})\,\mathbf{I}.
\end{equation}

Note that for $\lambda=0$ or $J=1$ without initial stresses, the hypoelastic material models with constitutive relations \eqref{2-47b} and \eqref{2-48} are rate forms of ones for hyperelastic materials represented in Table \ref{t3} \cite{KorobeynikovJElast2021}. However, of the entire family of hypoelastic models with constitutive relations \eqref{2-49}, only one model based on the logarithmic stress rate is the rate form of the Hencky isotropic hyperelastic  material model for an arbitrary type of deformation without initial stresses \cite{XiaoJElast1997,XiaoAM1997,XiaoAM1999,XiaoJElast1999} (see also \cite{KorobeynikovAAM2020,KorobeynikovJElast2021}).

Explicit expressions for objective stress rates and identification  of the considered isotropic Hooke-like hypoelasticity models with constitutive relations of the form
\begin{equation*}
    \bar{\boldsymbol{\tau}}^{\nabla} = 2\mu\, \widehat{\mathbf{D}} + \lambda\, (\text{tr}\,\widehat{\mathbf{D}})\,\mathbf{I}  \quad \Leftrightarrow \quad \boldsymbol{\tau}^{\nabla} = 2\mu\, \mathbf{d} + \lambda\, (\text{tr}\,\mathbf{d})\,\mathbf{I}.
\end{equation*}
are given in Table \ref{t5}.
\begin{table}
\caption{Explicit expressions for objective stress rates and identification of isotropic Hooke-like hypoelasticity models}
\label{t5}
\scriptsize
\begin{tabular}{l|lll}
\hline\noalign{\smallskip}
 Model's   &  \multicolumn{2}{c}{Definition of stress rates}                                                                   \\
 name      &  Lagrangian version ($\bar{\boldsymbol{\tau}}^{\nabla}$)    & Eulerian version ($\boldsymbol{\tau}^{\nabla}$)       \\
\noalign{\smallskip}\hline\noalign{\smallskip}
  Hypo-A   &  $\bar{\boldsymbol{\tau}}^{\sharp}\ \ \ \equiv \dot{\bar{\boldsymbol{\tau}}}-\bar{\boldsymbol{\tau}}\cdot\mathbf{r}_U^T - \mathbf{r}_U\cdot \bar{\boldsymbol{\tau}}$ & $\boldsymbol{\tau}^{\sharp}\ \ \  \equiv \dot{\boldsymbol{\tau}}-\boldsymbol{\tau}\cdot\boldsymbol{\ell}^T - \boldsymbol{\ell}\cdot \boldsymbol{\tau}$ \\
  Hypo-B   &  $\bar{\boldsymbol{\tau}}^{\flat}\ \ \ \equiv \dot{\bar{\boldsymbol{\tau}}}+ \bar{\boldsymbol{\tau}}\cdot\mathbf{r}_U + \mathbf{r}_U^T\cdot \bar{\boldsymbol{\tau}}$ & $\boldsymbol{\tau}^{\flat}\ \ \  \equiv \dot{\boldsymbol{\tau}}+ \boldsymbol{\tau}\cdot\boldsymbol{\ell} + \boldsymbol{\ell}^T\cdot \boldsymbol{\tau}$ \\
  Hypo-ZJ  &  $\bar{\boldsymbol{\tau}}^{\text{ZJ}}\ \equiv \dot{\bar{\boldsymbol{\tau}}} + \bar{\boldsymbol{\tau}}\cdot \mathbf{W} - \mathbf{W}\cdot \bar{\boldsymbol{\tau}}$ &
  $\boldsymbol{\tau}^{\text{ZJ}}\ \equiv \dot{\boldsymbol{\tau}}+ \boldsymbol{\tau}\cdot\mathbf{w} - \mathbf{w}\cdot \boldsymbol{\tau}$  \\
  Hypo-GN  &  $\bar{\boldsymbol{\tau}}^{\nabla}\ \ \equiv \dot{\bar{\boldsymbol{\tau}}}$ &
  $\boldsymbol{\tau}^{\text{GN}} \equiv \dot{\boldsymbol{\tau}}+ \boldsymbol{\tau}\cdot\boldsymbol{\omega}^R - \boldsymbol{\omega}^R\cdot \boldsymbol{\tau}$  \\
  Hypo-GS  &  $\bar{\boldsymbol{\tau}}^{\text{GS}} \, \equiv \dot{\bar{\boldsymbol{\tau}}} + \bar{\boldsymbol{\tau}}\cdot \boldsymbol{\Omega}^{\text{GS}} - \boldsymbol{\Omega}^{\text{GS}}\cdot \bar{\boldsymbol{\tau}}$ &
  $\boldsymbol{\tau}^{\text{GS}}\, \equiv \dot{\boldsymbol{\tau}}+ \boldsymbol{\tau}\cdot\boldsymbol{\omega}^{\text{GS}} - \boldsymbol{\omega}^{\text{GS}}\cdot \boldsymbol{\tau}$  \\
  Hypo-log  &  $\bar{\boldsymbol{\tau}}^{\log}\, \equiv \dot{\bar{\boldsymbol{\tau}}} + \bar{\boldsymbol{\tau}}\cdot \boldsymbol{\Omega}^{\log} - \boldsymbol{\Omega}^{\log}\cdot \bar{\boldsymbol{\tau}}$ &
  $\boldsymbol{\tau}^{\log}\, \equiv \dot{\boldsymbol{\tau}}+ \boldsymbol{\tau}\cdot\boldsymbol{\omega}^{\log} - \boldsymbol{\omega}^{\log}\cdot \boldsymbol{\tau}$  \\
\noalign{\smallskip}\hline
\end{tabular}
\end{table}

\section{Kinematics of finite simple shear deformations}
\label{sec:3}

We consider a homogeneous deformation of a material sample in space with the law of motion \eqref{1-6}. As noted in Section \ref{sec:1}, Thiel et al. \cite{ThielIJNLM2019} proposed to supplement the constraint \eqref{1-7} on the motion parameters $a$, $b$, $c$, and $d$ by two additional constraints
\begin{equation}\label{3-1}
  a\,c=1, \quad d=1,
\end{equation}
which ensure that the deformation of the sample is both planar ($d=1$) and isochoric ($J=a\,c\,d=1$).

Subject only to constraints \eqref{3-1}, the left and right Cauchy--Green deformation tensors \eqref{2-9} take the form
\begin{equation}\label{3-2}
  \mathbf{c}_{3D}=\left[
                    \begin{array}{ccc}
                      a^2+b^2   & b\,c  & 0 \\
                      b\,c      & c^2 & 0 \\
                      0         & 0   & 1 \\
                    \end{array}
                  \right],\quad\quad
  \mathbf{C}_{3D}= \left[
                    \begin{array}{ccc}
                      a^2 & a\,b      & 0 \\
                      a\,b  & b^2+c^2 & 0 \\
                      0   & 0       & 1 \\
                    \end{array}
                  \right].
\end{equation}

Since, in linear elasticity theory, Hooke's law for isotropic materials relates the coaxial Cauchy stress $\boldsymbol{\sigma}$ and infinitesimal strain  tensors $\boldsymbol{\varepsilon}$, pure shear stress and pure shear strain states for such (linear elastic) materials occur simultaneously (see, e.g., \cite{ThielIJNLM2019}). In finite strain theory, the coaxially of stress and strain tensors in an arbitrary deformation is ensured only for isotropic Cauchy/Green elastic materials (see, e.g., \cite{Ogden1984}). In this case, the Cauchy $\boldsymbol{\sigma}$ and Kirchhoff $\boldsymbol{\tau}$ stress tensors are coaxial with the left Cauchy--Green deformation tensor $\mathbf{c}$, and the rotated Cauchy $\bar{\boldsymbol{\sigma}}\equiv \mathbf{R}\cdot \boldsymbol{\sigma}\cdot \mathbf{R}^T$ and Kirchhoff $\bar{\boldsymbol{\tau}}\equiv \mathbf{R}\cdot \boldsymbol{\tau}\cdot \mathbf{R}^T$ stress tensors are coaxial with the right Cauchy--Green deformation tensor $\mathbf{C}$.

Further, along with the Eulerian pure shear stress \eqref{1-1}, we will also consider the Lagrangian pure shear stress
\begin{equation}\label{3-3}
   \bar{\boldsymbol{\sigma}}=\bar{\sigma}_{12}\mathbf{k}_1\otimes \mathbf{k}_2 + \bar{\sigma}_{12}\mathbf{k}_2\otimes \mathbf{k}_1\quad  \Leftrightarrow \quad
   \bar{\boldsymbol{\sigma}} = \left[                                                                                                                                                                   \begin{array}{ccc}
      0                 & \bar{\sigma}_{12} & 0 \\
      \bar{\sigma}_{12} & 0                 & 0 \\
      0                 & 0                 & 0 \\
   \end{array}
   \right].
\end{equation}
Since the stress tensors $\boldsymbol{\sigma}$ and $\bar{\boldsymbol{\sigma}}$ are related by rotation operations, the law of motion \eqref{1-6}, which for a sample of some material leads to the Eulerian pure shear stress \eqref{1-1}, generally does not lead to the Lagrangian pure shear stress \eqref{3-3} for this sample and vice versa. This implies that, in the general case of finite strains, the Eulerian and Lagrangian pure shear stresses (if they are possible at all for the law of motion \eqref{1-6} with constraints \eqref{3-1}) are reached for the law of motion \eqref{1-6} with different sets of the parameters $a$, $b$, and $c$ for $d=1$.

If the Eulerian pure shear stress \eqref{1-1} can be reached, the condition of coaxiality of the tensor $\mathbf{c}$ and the tensor $\boldsymbol{\sigma}$ leads to the equality of the components $c_{11}^{3D}=c_{22}^{3D}$ in $\eqref{3-2}_1$ \cite{MoonARMA1974}, which implies the constraint \eqref{1-7} on the parameters $a$, $b$, and $c$. Similarly, if the Lagrangian pure shear stress \eqref{3-3} can be reached, the coaxiality of the tensor $\mathbf{C}$ and the tensor $\bar{\boldsymbol{\sigma}}$ leads to the equality of the components $C_{11}^{3D}=C_{22}^{3D}$ in $\eqref{3-2}_2$ \cite{LinZAMP2024}, which implies the constraint \eqref{1-8} on the parameters $a$, $b$, and $c$.

To recapitulate, Thiel et al. \cite{ThielIJNLM2019} impose two common constraints \eqref{3-1} and two different constraints \eqref{1-7} or \eqref{1-8} on the four parameters $a$, $b$, $c$, and $d$ in the law of motion \eqref{1-6}. With the use of constraints \eqref{3-1} and \eqref{1-7} on the four parameters $a$, $b$, $c$, and $d$, the law of motion \eqref{1-6} for isotropic Cauchy/Green elastic materials allows the possibility of the occurrence of the Eulerian pure shear stress \eqref{1-1}. Similarly, with the adoption of constraints \eqref{3-1} and  \eqref{1-8} on the four parameters $a$, $b$, $c$, and $d$, the law of motion \eqref{1-6} for isotropic Cauchy/Green elastic materials allows the possibility of the occurrence of the Lagrangian pure shear stress \eqref{3-3}.

Note that each of the two laws of motion \eqref{1-6} with the sets of constraints \eqref{3-1} and \eqref{1-7} or \eqref{3-1} and \eqref{1-8} is a one-parameter law. Thiel et al. \cite{ThielIJNLM2019} have proposed a single dimensionless parameter $\alpha \in \mathds{R}$ as such a parameter and the following expressions for the quantities $a$, $b$, and $c$ ($d=1$):
\begin{description}
  \item[(i)] under the constraints $\eqref{3-1}_1$ and \eqref{1-7},
\begin{equation}\label{3-4}
  a(\alpha):= \frac{1}{\sqrt{\cosh (2\alpha)}},\quad\quad b(\alpha):= \frac{\sinh (2\alpha)}{\sqrt{\cosh (2\alpha)}},\quad\quad c(\alpha):= \sqrt{\cosh (2\alpha)};
\end{equation}
  \item[(ii)] under the constraints $\eqref{3-1}_1$ and \eqref{1-8},
\begin{equation}\label{3-5}
  a(\alpha):= \sqrt{\cosh (2\alpha)},\quad\quad b(\alpha):= \frac{\sinh (2\alpha)}{\sqrt{\cosh (2\alpha)}},\quad\quad c(\alpha):= \frac{1}{\sqrt{\cosh (2\alpha)}}.
\end{equation}
\end{description}
Following \cite{ThielIJNLM2019}, homogeneous deformations of samples with the law of motion \eqref{1-6} for $d=1$ and the sets of the parameters $a$, $b$, and $c$ \eqref{3-4} and \eqref{3-5} will be called \emph{left finite simple shear} (LFSS) and \emph{right finite simple shear} (RFSS) deformations, respectively. Configurations of an initially cubic sample with sides of length 1 for $\alpha = 0.5$ deformed in this way under plane strain conditions are shown in Fig. \ref{f4}.
\begin{figure}
\begin{center}
\includegraphics{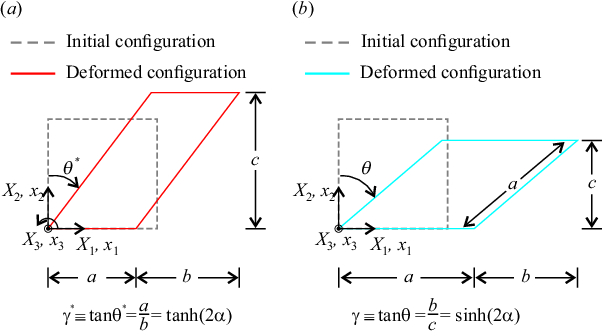}
\end{center}
\caption{ LFSS (\emph{a}) and RFSS (\emph{b}) deformations of an initially cubic sample with sides of length 1 under plane strain conditions for $\alpha = 0.5$.}
\label{f4}
\end{figure}
LFSS and RFSS deformations of an initially cubic sample with sides of length 1 under plane strain conditions for $\alpha = 0.5,\,1.0,\,1.5$ are shown in Fig. \ref{f5}.
\begin{figure}
\begin{center}
\includegraphics{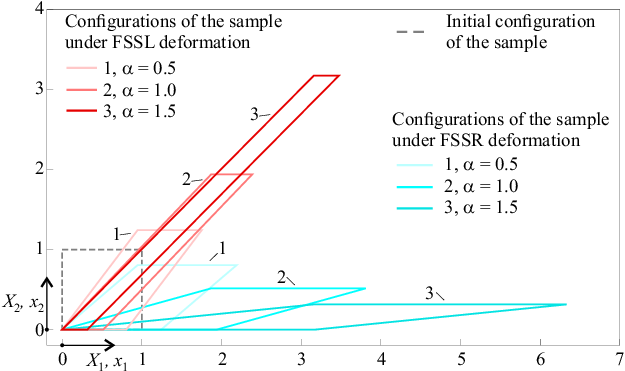}
\end{center}
\caption{ LFSS (\emph{a}) and RFSS (\emph{b}) deformations of an initially cubic sample with sides of length 1 under plane strain conditions for $\alpha = 0.5,\,1.0,\,1.5$}
\label{f5}
\end{figure}
We determine the shear angles $\theta^{\ast}$ and $\theta$ (see Fig. \ref{f4}) and their corresponding standard shear parameters $\gamma^{\ast}$ and $\gamma$
\begin{equation*}
  \gamma^{\ast} \equiv \tan \theta^{\ast}=\frac{b}{c} = \tanh (2\alpha),\quad\quad \gamma \equiv \tan \theta=\frac{b}{c} = \sinh (2\alpha).
\end{equation*}
Values of the parameters $a$, $b$, and $c$, the shear angles $\theta^{\ast}$ and $\theta$ and their corresponding standard shear parameters $\gamma^{\ast}$ and $\gamma$ are given in Tables \ref{t6} and \ref{t7}.
\begin{table}
\caption{Values of the parameters $a$, $b$, and $c$, the shear angle $\theta^{\ast}$, and their corresponding standard shear parameter $\gamma^{\ast}$ for LFSS deformation}
\label{t6}
\scriptsize
\begin{tabular}{l|lllll}
\hline\noalign{\smallskip}
 $\alpha$ [-]  &  $\theta^{\ast}$ [deg.]     & $\gamma^{\ast}$ [-] & $a$ [-]  & $b$ [-]   &  $c$ [-] \\
\noalign{\smallskip}\hline\noalign{\smallskip}
   0.5         & 37.29                       & 0.7616              & 0.8050   & 0.9461    & 1.2422   \\
   1.0         & 43.95                       & 0.9640              & 0.5156   & 1.8699    & 1.9396   \\
   1.5         & 44.86                       & 0.9951              & 0.3152   & 3.1573    & 3.1730   \\
\noalign{\smallskip}\hline
\end{tabular}
\end{table}

\begin{table}
\caption{Values of the parameters $a$, $b$, and $c$, the shear angle $\theta$, and their corresponding standard shear parameter $\gamma$ for RFSS deformation}
\label{t7}
\scriptsize
\begin{tabular}{l|lllll}
\hline\noalign{\smallskip}
 $\alpha$ [-]  &  $\theta$ [deg.] & $\gamma$ [-] & $a$ [-]  & $b$ [-]  &  $c$ [-] \\
\noalign{\smallskip}\hline\noalign{\smallskip}
   0.5         &  49.61           & 1.1752       & 1.2422   & 0.9461    & 0.8050  \\
   1.0         &  74.59           & 3.6269       & 1.9396   & 1.8699    & 0.5156  \\
   1.5         &  84.30           & 10.018       & 3.1730   & 3.1573    & 0.3152  \\
\noalign{\smallskip}\hline
\end{tabular}
\end{table}

Since $\cosh (2 \alpha)\geq 1$ for any values of the parameter $\alpha$, it follows from \eqref{3-4} and \eqref{3-5} that the inequalities $a\leq 1$ and $c\geq 1$ are valid for LFSS deformations and the inequalities $a\geq 1$ and $c\leq 1$ hold for RFSS deformations, which agrees with the data presented in Tables \ref{t6} and \ref{t7}. In addition, the shear angles $\theta^{\ast}$ and $\theta$ and their corresponding standard shear parameters $\gamma^{\ast}$ and $\gamma$ have the following limiting values:
\begin{equation}\label{3-6a}
  \lim \limits_{\alpha \to \infty} \theta^{\ast}(\alpha)= 45^{\circ},\quad \lim \limits_{\alpha \to \infty} \gamma^{\ast}(\alpha) = 1, \quad
  \lim \limits_{\alpha \to \infty} \theta(\alpha)= 90^{\circ},\quad \lim \limits_{\alpha \to \infty} \gamma(\alpha) = \infty.
\end{equation}

Note that for simple shear deformations, it is convenient to use as the shear parameter the quantity $\gamma^s\equiv \tan \theta$ (see Fig. \ref{f2},\emph{b}), which varies similarly to the  parameter $\gamma$ in RFSS deformation. Unfortunately, for LFSS deformations, the quantity $\gamma^{\ast}$ cannot be used as the shear parameter in the same ranges of variation (see Table \ref{t6} and Eq. $\eqref{3-6a}_2$). In this paper, we choose the parameter $\alpha\in [0,\,\infty)$ as the shear parameter and use only its non-negative values, since negative values do not lead to any fundamentally new solutions. To maintain consistency of the procedure of testing samples for finite simple shear and simple shear deformations, we choose the main range of the parameter $\alpha\in [0,\,1.5]$, since for $\alpha=1.5$,  the value of the parameter $\gamma\approx 10$ for RFSS deformations (see Table \ref{t7}). We used a similar value for the parameter $\gamma^s\in [0,\,10]$ in testing samples under simple shear deformation conditions in previous studies (see, e.g., \cite{Korobeynikov2023,KorobeynikovZAMM2024}).

The law of motion \eqref{1-6} with $d=1$ corresponds to 2D plane strain deformation. Since this law of motion subject to expressions \eqref{3-4} and \eqref{3-5} corresponds to isochoric deformation, it follows that for the completely imposed kinematics in the form of the considered law of motion for samples of even compressible material, the Cauchy stress tensor components are determined, as for incompressible material samples, with accuracy up to an arbitrary spherical constituent of this tensor. To uniquely determine the Cauchy stress tensor components, it is required to determine this constituent by specifying additional boundary conditions for the normal components of the Cauchy stress tensor in some regions of the sample surface. For example, for a simple shear deformation of a sample, it is assumed that, along with plane strain conditions, either a plane stress state occurs in the sample, or the normal pressure on the lateral side of the sample is specified (see, e.g., \cite{HorganJElast2010}). In this study, we only use material models for which isochoric plane-strain deformation simultaneously satisfy the plane stress condition, i.e.,
\begin{equation*}
  \sigma_{33} = \bar{\sigma}_{33}=0.
\end{equation*}
We use these conditions and obtain unique values of the Cauchy stress tensor components for the finite simple shear deformations considered. Therefore, below we consider only 2D plane strain, which significantly simplifies further calculations.

Following \cite{ThielIJNLM2019}, we obtain two expressions of the deformation gradient $\mathbf{F}$ for LFSS and RFSS deformations using the expressions \eqref{3-4} and \eqref{3-5}
\begin{equation}\label{3-8}
  \mathbf{F}_L = \frac{1}{\sqrt{\cosh (2\alpha)}}\left[
                                                   \begin{array}{cc}
                                                     1 & \sinh (2\alpha) \\
                                                     0 & \cosh (2\alpha) \\
                                                   \end{array}
                                                 \right],\quad
  \mathbf{F}_R = \frac{1}{\sqrt{\cosh (2\alpha)}}\left[
                                                   \begin{array}{cc}
                                                     \cosh (2\alpha) & \sinh (2\alpha) \\
                                                            0        &      1          \\
                                                   \end{array}
                                                 \right],
\end{equation}
which, using the polar decomposition \eqref{2-5}, can be represented as  (cf., \cite{NeffIJES2014})
\begin{equation}\label{3-9}
  \mathbf{F}_L = \mathbf{V}_L\cdot \mathbf{R},\quad\quad \mathbf{F}_R = \mathbf{R} \cdot \mathbf{U}_R,
\end{equation}
where $\mathbf{V}_L$ and $\mathbf{U}_R$ are the symmetric positive definite left (Eulerian) and right (Lagrangian) stretch tensors for LFSS and RFSS deformations, respectively, and $\mathbf{R}$ is the polar rotation tensor, identical for both decompositions \eqref{3-9}. Expression for these tensors are presented in \cite{ThielIJNLM2019}
\begin{equation}\label{3-10}
  \mathbf{V}_L = \mathbf{U}_R = \left[
                                                   \begin{array}{cc}
                                                     \cosh (\alpha) & \sinh (\alpha) \\
                                                     \sinh (\alpha) & \cosh (\alpha) \\
                                                   \end{array}
                                                 \right],\quad
  \mathbf{R} = \frac{1}{\sqrt{\cosh (2\alpha)}}\left[
                                                   \begin{array}{cc}
                                                     \cosh (\alpha)  & \sinh (\alpha) \\
                                                     -\sinh (\alpha) & \cosh (\alpha) \\
                                                   \end{array}
                                                 \right].
\end{equation}

Since, for finite simple shear deformations, the two principal stretches are different when $\alpha \neq 0$, the stretch tensors $\mathbf{V}_L$ and $\mathbf{U}_R$ can be represented  in classical spectral form with a unique choice of eigenvectors (corresponding to the principal directions)
\begin{equation*}
  \mathbf{V}_L = \lambda_1 \mathbf{n}_1\otimes \mathbf{n}_1 + \lambda_2 \mathbf{n}_2\otimes \mathbf{n}_2,\quad\quad \mathbf{U}_R = \lambda_1 \mathbf{N}_1\otimes \mathbf{N}_1 + \lambda_2 \mathbf{N}_2\otimes \mathbf{N}_2,
\end{equation*}
where $\lambda_1$ and $\lambda_2$ are the eigenvalues (principal stretches) and $\mathbf{n}_1$, $\mathbf{n}_2$ and $\mathbf{N}_1$, $\mathbf{N}_2$ are the subordinate eigenvectors of the tensors $\mathbf{V}_L$ and $\mathbf{U}_R$, respectively \cite{LinZAMP2024,ThielIJNLM2019}
\begin{equation}\label{3-12}
  \lambda_1 = \mathrm{e}^{\alpha},\quad\quad \lambda_2 = \mathrm{e}^{-\alpha},\quad\quad \mathbf{n}_1=\mathbf{N}_1 = \frac{\sqrt{2}}{2}\left[
                                                                                                                               \begin{array}{c}
                                                                                                                                 1 \\
                                                                                                                                 1 \\
                                                                                                                               \end{array}
                                                                                                                             \right],\quad\quad
   \mathbf{n}_2=\mathbf{N}_2 = \frac{\sqrt{2}}{2}\left[
                                                  \begin{array}{c}
                                                    -1 \\
                                                    1  \\
                                                 \end{array}
                                                \right].
\end{equation}
Similarly, the stretch tensors $\mathbf{V}_L$ and $\mathbf{U}_R$ can be represented in basis-free spectral form using the eigenprojections of these tensors
\begin{equation}\label{3-13}
  \mathbf{V}_L = \lambda_1 \mathbf{V}_1^L + \lambda_2 \mathbf{V}_2^L,\quad\quad\quad \mathbf{U}_R = \lambda_1 \mathbf{U}_1^R + \lambda_2 \mathbf{U}_2^R,
\end{equation}
where the eigenvalues $\lambda_1$ and $\lambda_2$ are given in \eqref{3-12} and the subordinate eigenprojections have the form
\begin{equation}\label{3-14}
  \mathbf{V}^L_1=\mathbf{U}^R_1 = \mathbf{n}_1 \mathbf{n}_1^T = \mathbf{N}_1 \mathbf{N}_1^T = \frac{1}{2}\left[
                                                                                                           \begin{array}{cc}
                                                                                                             1 & 1 \\
                                                                                                             1 & 1 \\
                                                                                                           \end{array}
                                                                                                         \right],\quad
\mathbf{V}^L_2=\mathbf{U}^R_2 = \mathbf{n}_2 \mathbf{n}_2^T = \mathbf{N}_2 \mathbf{N}_2^T = \frac{1}{2}\left[
                                                                                                           \begin{array}{cc}
                                                                                                             1  & -1 \\
                                                                                                             -1 &  1 \\
                                                                                                           \end{array}
                                                                                                         \right].
\end{equation}
Of course, the same expressions for the eigenprojections can be derived using Sylvester's formula \eqref{2-3}.

In what follows, we will seek material models for which finite simple shear deformations lead to Eulerian or Lagrangian pure shear stresses. The following two propositions  simplify the search for these material models.\\

\begin{proposition}
\label{Pr:3-1}
If an LFSS (RFSS) deformation of a sample of any material leads to an Eulerian (Lagrangian) pure shear stress of the form \eqref{1-1} (\eqref{3-3}), the same deformation cannot lead to a Lagrangian (Eulerian) pure stress of the form \eqref{3-3} (\eqref{1-1}).
\end{proposition}

\begin{proof}
We prove the proposition for LFSS deformations; for RFSS deformations, the proposition is proved similarly. Let the Cauchy stress tensor for the LFSS deformation have the form \eqref{1-1}, or, in the 2D representation, the form
\begin{equation*}
   \boldsymbol{\sigma} = \left[                                                                                                                                                                   \begin{array}{cc}
      0           & \sigma_{12}  \\
      \sigma_{12} & 0            \\
   \end{array}
   \right].
\end{equation*}
The rotated Cauchy stress tensor is given by
\begin{equation*}
  \bar{\boldsymbol{\sigma}}=\mathbf{R}^T\cdot \boldsymbol{\sigma}\cdot \mathbf{R}.
\end{equation*}
Using the expression $\eqref{3-10}_2$ for the tensor $\mathbf{R}$ and the trigonometric identities
\begin{equation*}
  2\sinh (\alpha)\cosh (\alpha)=\sinh (2\alpha),\quad\quad\quad   \cosh^2 (\alpha) - \sinh^2 (\alpha)=1,
\end{equation*}
we obtain the expression
\begin{equation}\label{3-18}
  \mathbf{R}^T\cdot \boldsymbol{\sigma}\cdot \mathbf{R} = \frac{\sigma_{12}}{\cosh (2\alpha)}\left[                                                                                                                                                                   \begin{array}{cc}
      -\sinh (2\alpha) & 1                \\
      1                & \sinh (2\alpha)  \\
   \end{array}
   \right].
\end{equation}
Since, for $\alpha >0$, we have $\sinh (2\alpha)>0$, it follows from \eqref{3-18} that for LFSS deformations corresponding to Eulerian pure shear stresses for $\alpha >0$, the following inequality holds:
\begin{equation*}
   \bar{\boldsymbol{\sigma}} \neq \left[                                                                                                                                                                   \begin{array}{cc}
      0           & \bar{\sigma}_{12}  \\
      \bar{\sigma}_{12} & 0            \\
   \end{array}
   \right];
\end{equation*}
hence the proposition is proved.
\end{proof}

\begin{proposition}
\label{Pr:3-2}
For finite simple shear deformations of a sample of any isotropic Cauchy/Green elastic material, the equality
\begin{equation}\label{3-20}
  \boldsymbol{\sigma}(\mathbf{V}_L)= \bar{\boldsymbol{\sigma}}(\mathbf{U}_R)
\end{equation}
holds for any $\alpha >0$.
\end{proposition}

\begin{proof}
For the material considered in the proposition, the equality (see, e.g., \cite{Ogden1984})
\begin{equation}\label{3-21}
  \boldsymbol{\sigma}=\mathbf{f}(\mathbf{V})\quad \Leftrightarrow \quad \bar{\boldsymbol{\sigma}}=\mathbf{f}(\mathbf{U})
\end{equation}
holds, where $\mathbf{f}$ is an isotropic tensor function of its tensorial argument. Since, for any $\alpha$, the expression $\eqref{3-10}_1$ is valid, it follows from equality \eqref{3-21} that equality \eqref{3-20} is also valid.
\end{proof}

\begin{corollary}
\label{Cor:3-2-1}
If, for any isotropic Cauchy/Green elastic material model, LFSS (RFSS) deformations lead to Eulerian (Lagrangian) pure shear stress states, then RFSS (LFSS) deformations lead to Lagrangian (Eulerian) pure shear stress states subject to the equality
\begin{equation}\label{3-22}
   \boldsymbol{\sigma}(\mathbf{V}_L)= \bar{\boldsymbol{\sigma}}(\mathbf{U}_R) = \left[                                                                                                                                                                   \begin{array}{cc}
      0   & s  \\
      s   & 0  \\
   \end{array}
   \right],
\end{equation}
where $s\equiv \sigma_{12}=\bar{\sigma}_{12}$.
\end{corollary}

\section{Testing some isotropic hyper-/hypo-elastic material models}
\label{sec:4}

In this Section, we test Hooke-like isotropic hyper-/hypo-elastic material models with the constitutive relations presented in Section \ref{sec:2-3} under the LFSS/RFSS deformations defined in Section \ref{sec:3}. Some Hill's linear isotropic hyperelastic material models are tested in Section \ref{sec:4-1}, and some one- or two-power Ogden's mode in Section \ref{sec:4-2}. Finally, some Hooke-like isotropic hypoelastic material models are tested in Section \ref{sec:4-3}.

\subsection{Testing some HLIH material models based on SP strain functions}
\label{sec:4-1}

\begin{proposition}
\label{Pr:4-1}
For any material model from the HLIH family based on SP strain tensors with a scale function $f(\lambda)$, LFSS and RFSS deformations lead to Eulerian and Lagrangian pure shear stresses, respectively, of the form \eqref{3-22}, where
\begin{equation}\label{4-1}
  s(\alpha)=2 \mu f(\lambda_1)f^{\prime}(\lambda_1)\lambda_1,\quad\quad \lambda_1=\mathrm{e}^{\alpha}.
\end{equation}
\end{proposition}

\begin{proof}
Since, for finite simple shear deformations, the equalities $\lambda_1=\mathrm{e}^{\alpha}$, $\lambda_2=\lambda_1^{-1}$ hold, the definition  of any SP strain function (cf., \cite{KorobeynikovJElast2019}) leads to
\begin{equation}\label{4-2}
  f(\lambda_2)=-f(\lambda_1),\quad\quad f(\lambda_3)=0,
\end{equation}
whence it follows that the first term on the right-hand side of $\eqref{2-43}_1$ vanishes. Substituting expression $\eqref{2-43}_1$ into $\eqref{2-42}_2$ and taking into account \eqref{4-2}, we obtain a 2D equality for the Cauchy stress tensor ($\boldsymbol{\tau}=\boldsymbol{\sigma}$ for the isochoric deformation used) under LFSS deformation:
\begin{equation}\label{4-3}
  \boldsymbol{\sigma}= 2 \mu [f(\lambda_1)f^{\prime}(\lambda_1)\lambda_1 \mathbf{V}_1^L + f(\lambda_2)f^{\prime}(\lambda_2)\lambda_2 \mathbf{V}_2^L].
\end{equation}
Using equality (6) from \cite{KorobeynikovJElast2019}
\begin{equation*}
  f^{\prime}(\lambda_1)\lambda_1=f^{\prime}(\lambda_2)\lambda_2,
\end{equation*}
which is valid for this material model, and equality $\eqref{4-2}_1$, from \eqref{4-3} we obtain the following expression for $\boldsymbol{\sigma}$:
\begin{equation}\label{4-5}
  \boldsymbol{\sigma}= 2 \mu f(\lambda_1)f^{\prime}(\lambda_1)\lambda_1\left[                                                                                                                  \begin{array}{cc}
      0   & 1  \\
      1   & 0  \\
   \end{array}
   \right] .
\end{equation}
Here we used the equality
\begin{equation}\label{4-6}
  \mathbf{V}_1^L - \mathbf{V}_2^L = \left[
   \begin{array}{cc}
      0   & 1  \\
      1   & 0  \\
   \end{array}
   \right] .
\end{equation}
Taking into account the expression for $\lambda_1$ from $\eqref{3-12}_1$ and comparing expression  \eqref{4-5} with the Eulerian version of expression \eqref{3-22}, we obtain equality \eqref{4-1}. The expression in \eqref{3-22} for the Lagrangian version follows from Corollary \ref{Cor:3-2-1}.
\end{proof}

We consider three HLIH models based on the SP strain tensors (the Hencky, Pelzer, and Mooney) generated by the scale functions $f^{(0)}(\lambda)$, $f^P(\lambda)$, and $f^M(\lambda)$ (see Table \ref{t1}). Following \cite{KorobeynikovJElast2019}, we call the corresponding models from the HLIH family the H, P, and M material models. The function $s(\alpha)$ is defined as follows:
\begin{itemize}
  \item for the H model (the Hencky isotropic hyperelastic material model).
  \begin{equation}\label{4-7}
    s(\alpha)= 2\mu\, \alpha,
  \end{equation},
  \item for the P model (the Pelzer isotropic hyperelastic material model),
  \begin{equation}\label{4-8}
    s(\alpha)= \frac{\mu}{2}(\mathrm{e}^{2\alpha}-\mathrm{e}^{-2\alpha}),
  \end{equation},
  \item for the M model (the Mooney isotropic hyperelastic material model),
  \begin{equation}\label{4-9}
    s(\alpha)= \frac{\mu}{4}(\mathrm{e}^{4\alpha}-\mathrm{e}^{-4\alpha}).
  \end{equation}.
\end{itemize}

Thus, the Cauchy stress tensor components for an Eulerian pure shear stress state under LFSS deformation are given by
\begin{equation*}
  \sigma_{12}(\alpha) = s(\alpha),\quad\quad \sigma_{11}=\sigma_{22}=0,
\end{equation*}
and the rotated Cauchy stress tensor components for an Lagrangian pure shear stress state under RFSS deformation are written similarly:
\begin{equation}\label{4-11}
  \bar{\sigma}_{12}(\alpha) = s(\alpha),\quad\quad \bar{\sigma}_{11}=\bar{\sigma}_{22}=0.
\end{equation}
To obtain the Cauchy stress tensor components for this deformation, it is required to rotate the stress tensor $\bar{\boldsymbol{\sigma}}$
\begin{equation*}
  \boldsymbol{\sigma}(\alpha)=\mathbf{R}(\alpha)\cdot \bar{\boldsymbol{\sigma}}(\alpha)\cdot \mathbf{R}^T(\alpha).
\end{equation*}
Using expressions for the tensors $\bar{\boldsymbol{\sigma}}(\alpha)$ in \eqref{3-22} and $\mathbf{R}(\alpha)$ in $\eqref{3-10}_2$, we obtain the Cauchy stress tensor components for RFSS deformations
\begin{equation}\label{4-13}
  \sigma_{12}(\alpha) = \frac{s(\alpha)}{\cosh (2\alpha)},\quad\quad \sigma_{11}(\alpha)= s(\alpha)\tanh (2\alpha),\quad\quad \sigma_{22}=-\sigma_{11},
\end{equation}
where the function $s(\alpha)$ is determined for each of the models (H, P, and M) from \eqref{4-7}-\eqref{4-9}.

Plots of $\sigma_{12}(\alpha)$  obtained for the H, P, and M models under LFSS deformation are shown in Fig. \ref{f6},\emph{a}, and plots of $\sigma_{12}(\alpha)$ and $\sigma_{11}(\alpha)$ obtained for the same models under RFSS deformation are shown in Fig. \ref{f6},\emph{b}.
\begin{figure}
\begin{center}
\includegraphics{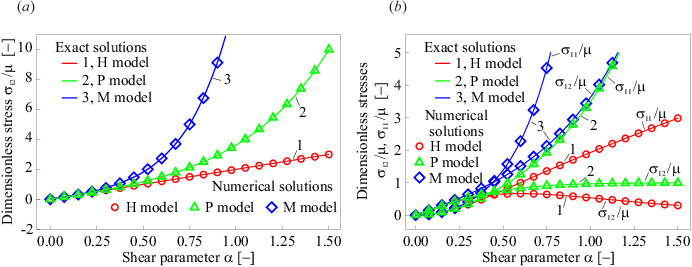}
\end{center}
\caption{Plots of the Cauchy stress tensor component $\sigma_{12}(\alpha)$ under LFSS deformation (\emph{a}) and the components $\sigma_{12}(\alpha)$ and $\sigma_{11}(\alpha)$ of the Cauchy stress tensor under RFSS deformation (\emph{b}) for the H, P, and M models (see Table \ref{t2}).}
\label{f6}
\end{figure}
The markers in the figures show the results of computer simulations for both types of deformations using the commercial FE system MSC.Marc, into which we implemented these three material models via Hypela2 user subroutine in a total Lagrangian formulation (for details, see \cite{KorobeynikovIJSS2022,KorobeynikovMTDM2024}).

\subsection{One- and two-power Ogden's models}
\label{sec:4-2}

We consider constitutive relations for the generalized Ogden's model \cite{KorobeynikovAAM2023} taking into account material compressibility that were obtained from constitutive relations for an incompressible material model \cite{KorobeynikovAAM2025} by replacing the term $-q\mathbf{I}$ ($q$ is the Lagrange multiplier) by the term $\lambda \log J\mathbf{I}$. Since, for isochoric deformation, $\log J=0$, it follows that constitutive relations for compressible material under this deformation can be formally derived from constitutive relations for incompressible material by equating the term $-q\mathbf{I}$ with the zero tensor, which can be done assuming that the plane stress and plane strain conditions for isochoric deformation are satisfied simultaneously. The latter is valid for the material models considered in this paper, since, for these models, the equality $\lambda_3=1$ (the plane strain condition) leads to the equality $\sigma_{33}=0$ (the plane stress condition). In view of the above and the fact that for the Cauchy and Kirchhoff stress tensors of isochoric deformation, $\boldsymbol{\tau}=\boldsymbol{\sigma}$ and $\bar{\boldsymbol{\tau}}=\bar{\boldsymbol{\sigma}}$, where $\bar{\boldsymbol{\sigma}}\equiv \mathbf{R}^T\cdot \boldsymbol{\sigma}\cdot \mathbf{R}$. Thus, for finite simple shear deformations, constitutive relations for the one-power Ogden's models (see Table \ref{t3}) are written as follows:
\begin{description}
  \item[(the Ogden-A model)] for the Ogden-A model,
\begin{equation}\label{4-14}
    \bar{\boldsymbol{\sigma}} = 2\mu\, \mathbf{E}^{(2)} \quad \Leftrightarrow \quad \boldsymbol{\sigma} = 2\mu\, \mathbf{e}^{(2)}.
\end{equation}
  \item[(the Ogden-B model)] for the Ogden-B model,
\begin{equation}\label{4-15}
    \bar{\boldsymbol{\sigma}} = 2\mu\, \mathbf{E}^{(-2)} \quad \Leftrightarrow \quad \boldsymbol{\sigma} = 2\mu\, \mathbf{e}^{(-2)}.
\end{equation}
\end{description}
Hereinafter, $\mathbf{E}^{(2)}$ and $\mathbf{E}^{(-2)}$ are Lagrangian strain tensors and $\mathbf{e}^{(2)}$ and $\mathbf{e}^{(-2)}$ are Eulerian ones (see Table~\ref{t1}).

We first test these models under LFSS deformation using Eulerian versions of constitutive relations. The expressions for the tensors $\mathbf{e}^{(2)}$ and $\mathbf{e}^{(-2)}$ given in Table~\ref{t1} lead to the equalities
\begin{equation}\label{4-16}
  \mathbf{e}^{(2)} = \frac{1}{2}(\mathbf{c}-\mathbf{I}),\quad\quad \mathbf{e}^{(-2)} = \frac{1}{2}(\mathbf{I}- \mathbf{c}^{-1}),
\end{equation}
where the tensors $\mathbf{c}$ and $\mathbf{c}^{-1}$ are defined in $\eqref{2-9}_2$ and $\eqref{2-10}_2$.

We define the tensor $\mathbf{c}$ using the expressions for the eigenprojections of the tensor $\mathbf{V}_L$ in \eqref{3-14}
\begin{equation*}
 \mathbf{c} = \mathbf{V}_L^2=\lambda_1^2\,\mathbf{V}^L_1 + \lambda_2^2\,\mathbf{V}^L_2 = \frac{\mathrm{e}^{2\alpha}}{2}\left[
                                                                                                                         \begin{array}{cc}
                                                                                                                           1 & 1 \\
                                                                                                                           1 & 1 \\
                                                                                                                         \end{array}
                                                                                                                       \right]
  + \frac{\mathrm{e}^{-2\alpha}}{2}\left[
                                        \begin{array}{cc}
                                          1 & -1 \\
                                          -1 & 1 \\
                                        \end{array}
                                      \right],
\end{equation*}
or
\begin{equation}\label{4-18}
 \mathbf{c} = \left[
                      \begin{array}{cc}
                        \cosh (2\alpha) & \sinh (2\alpha) \\
                        \sinh (2\alpha) & \cosh (2\alpha) \\
                      \end{array}
                    \right].
\end{equation}
Similarly, we obtain
\begin{equation*}
 \mathbf{c}^{-1} = \mathbf{V}_L^{-2}=\lambda_1^{-2}\,\mathbf{V}^L_1 + \lambda_2^{-2}\,\mathbf{V}^L_2 = \frac{\mathrm{e}^{-2\alpha}}{2}\left[
                                                                                                                         \begin{array}{cc}
                                                                                                                           1 & 1 \\
                                                                                                                           1 & 1 \\
                                                                                                                         \end{array}
                                                                                                                       \right]
  + \frac{\mathrm{e}^{2\alpha}}{2}\left[
                                        \begin{array}{cc}
                                          1 & -1 \\
                                          -1 & 1 \\
                                        \end{array}
                                      \right],
\end{equation*}
or
\begin{equation}\label{4-20}
 \mathbf{c}^{-1} = \left[
                      \begin{array}{cc}
                        \cosh (2\alpha)  & -\sinh (2\alpha) \\
                        -\sinh (2\alpha) &  \cosh (2\alpha) \\
                      \end{array}
                    \right].
\end{equation}

Using equalities \eqref{4-16}, \eqref{4-18}, and \eqref{4-20}, from \eqref{4-14} and \eqref{4-15} we obtain the following expressions for the components of the Cauchy stress tensor for the tested material models under LFSS deformation:
\begin{description}
    \item[(the Ogden-A model)] for the Ogden-A model,
\begin{equation}\label{4-21}
    \sigma_{12}(\alpha) = \mu \sinh (2\alpha),\quad\quad \sigma_{11}(\alpha) = \sigma_{22}(\alpha) = \mu[\cosh (2\alpha) - 1];
\end{equation}
   \item[(the Ogden-B model)] for the Ogden-B model,
\begin{equation}\label{4-22}
    \sigma_{12}(\alpha) = \mu \sinh (2\alpha),\quad\quad \sigma_{11}(\alpha) = \sigma_{22}(\alpha) = \mu[1 - \cosh (2\alpha)].
\end{equation}
\end{description}

The components $\bar{\sigma}_{ij}$ ($i,\,j=1,\,2$) for these material models under RFSS deformation are obtained using the statements of Proposition \ref{Pr:3-2},
\begin{equation*}
  \bar{\sigma}_{ij}(\alpha)=\sigma_{ij}(\alpha)\ (i,\,j=1,\,2),
\end{equation*}
where expressions for the quantities $\sigma_{ij}$ ($i,\,j=1,\,2$) are given in \eqref{4-21}, \eqref{4-22}. Applying rotation transformations, we obtain the components $\sigma_{ij}$ ($i,\,j=1,\,2$) of the Cauchy stress tensor for the Ogden-A and Ogden-B material models under RFSS deformation:
\begin{description}
  \item[(the Ogden-A model)] for the Ogden-A model,
\begin{equation*}
    \sigma_{12}(\alpha) = \mu \tanh (2\alpha),\quad \sigma_{11}(\alpha) = \mu\left[\frac{\cosh (4\alpha)}{\cosh (2\alpha)}-1\right],\quad
    \sigma_{22}(\alpha) = \mu \left[\frac{1}{\cosh (2\alpha)}-1\right];
\end{equation*}
  \item[(the Ogden-B model)] for the Ogden-B model,
\begin{equation*}\label{4-25}
    \sigma_{12}(\alpha) = \mu \tanh (2\alpha),\quad \sigma_{11}(\alpha) = \mu\left[1-\frac{1}{\cosh (2\alpha)}\right],\quad
    \sigma_{22}(\alpha) = \mu\left[1-\frac{\cosh (4\alpha)}{\cosh (2\alpha)}\right].
\end{equation*}
\end{description}

Plots of $\sigma_{12}(\alpha)$ and $\sigma_{11}(\alpha)=\sigma_{22}(\alpha)$ obtained for the Ogden-A,B models under LFSS deformation are shown in Fig. \ref{f7},(\emph{a}), and  plots of $\sigma_{12}(\alpha)$, $\sigma_{11}(\alpha)$, and $\sigma_{22}(\alpha)$ obtained  for the same models under RFSS deformation are shown in Fig. \ref{f7},(\emph{b}).
\begin{figure}
\begin{center}
\includegraphics{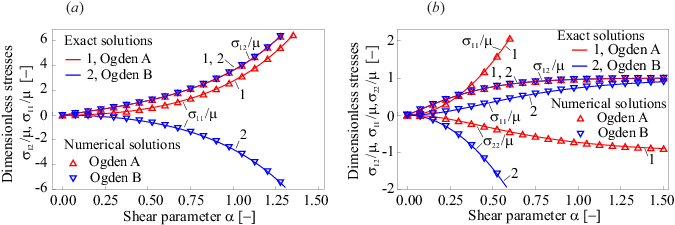}
\end{center}
\caption{Plots of the Cauchy stress tensor components for the Ogden-A,B models (see Table \ref{t3}) under LFSS (\emph{a}) and RFSS (\emph{b}) deformations.}
\label{f7}
\end{figure}

Based on the obtained expressions for the components of the Cauchy stress tensor for the Ogden-A,B models under LFSS and RFSS deformations, it can be concluded that, for these models, finite simple shear deformations lead neither to Eulerian nor to Lagrangian pure shear stresses. This assertion for the Ogden-A (neo-Hookean) model agrees with the similar assertion made for the same material model in \cite{ThielIJNLM2019}. Apparently, the assertion made here for two one-power Ogden type models is also valid for all material models of this type, except for the Hencky isotropic material model (i.e., for any of these models with the parameter $n\neq 0$).

Consider the following two-power Ogden's models (taking into account that the term $\lambda\, (\log J)\,\mathbf{I}=0$ for isochoric LFSS and RFSS deformations) for $n\neq 0$ (see Table \ref{t4}):
\begin{equation}\label{4-26}
    \bar{\boldsymbol{\sigma}} = \mu\, (\mathbf{E}^{(n)} +  \mathbf{E}^{(-n)})\quad \Leftrightarrow \quad \boldsymbol{\sigma} = \mu\, (\mathbf{e}^{(n)} + \mathbf{e}^{(-n)}).
\end{equation}
Since the expressions in parentheses on the right-hand sides of \eqref{4-26} are nothing more than doubled expressions for the Ba\v{z}ant--Itskov strain tensors belonging  to the family of SP strain tensors and generated by the scale function \eqref{2-13}, constitutive relations \eqref{4-26} can be rewritten as
\begin{equation}\label{4-27}
    \bar{\boldsymbol{\sigma}} = 2\mu\, \mathbf{E}_n\quad \Leftrightarrow \quad \boldsymbol{\sigma} = 2\mu\, \mathbf{e}_n,
\end{equation}
where
\begin{equation*}
    \mathbf{E}_n\equiv \frac{1}{2}(\mathbf{E}^{(n)} +  \mathbf{E}^{(-n)})\quad \Leftrightarrow \quad \mathbf{e}_n\equiv \frac{1}{2}(\mathbf{e}^{(n)} +  \mathbf{e}^{(-n)}).
\end{equation*}

We first test OBI models under LFSS deformation using Eulerian versions of their constitutive relations. From the expressions for the tensors $\mathbf{e}_n$ with scale functions from \eqref{2-13}, we obtain ($f_n(\lambda_3)=0$)
\begin{equation*}
  \mathbf{e}_n = f_n(\lambda_1)\,\mathbf{V}_1^L + f_n(\lambda_2)\,\mathbf{V}_2^L.
\end{equation*}
Equalities $\eqref{4-2}_1$ and \eqref{4-6} lead to the expression
\begin{equation}\label{4-30}
  \mathbf{e}_n = f_n(\lambda_1)\left[
   \begin{array}{cc}
      0   & 1  \\
      1   & 0  \\
   \end{array}
   \right] .
\end{equation}

Using expressions \eqref{4-27} and \eqref{4-30}, we find that for the OBI models, LFSS deformations lead to the Eulerian pure shear stress \eqref{3-22}, where ($n\neq 0$)
\begin{equation}\label{4-31}
  s(\alpha) = 2 \mu\ f_n(\lambda_1) = \frac{\mu}{n}(\mathrm{e}^{\alpha n} - \mathrm{e}^{-\alpha n}).
\end{equation}

In particular, we obtain expressions for the quantity $s(\alpha)$ for the following models:
\begin{itemize}
  \item for the OBI-P model (with the parameter $n=1$)
\begin{equation*}
  s(\alpha) = \mu (\mathrm{e}^{\alpha} - \mathrm{e}^{-\alpha});
\end{equation*}
  \item for the OBI-M model (with the parameter $n=2$)
\begin{equation}\label{4-33}
  s(\alpha) =  \frac{\mu}{2}(\mathrm{e}^{2 \alpha} - \mathrm{e}^{-2 \alpha}).
\end{equation}
\end{itemize}
Note that the expression for the quantity $s(\alpha)$ in \eqref{4-33} for the OBI-M model coincides with the expression \eqref{4-8} for the Pelzer model from the HLIH family.

The components of the Cauchy stress tensor for these material models tested under RFSS deformations have the form \eqref{4-13}, where the quantity $s(\alpha)$ is defined in \eqref{4-31}.

Plots of $\sigma_{12}(\alpha)$ obtained for the Hencky, OBI-P, and OBI-M models under LFSS deformation are shown in Fig. \ref{f8},\emph{a}, and plots of $\sigma_{12}(\alpha)$ and $\sigma_{11}(\alpha)$ obtained for the same models under RFSS deformation are shown in Fig. \ref{f8},\emph{b}.
\begin{figure}
\begin{center}
\includegraphics{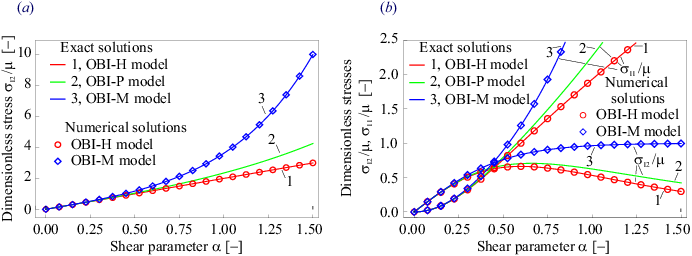}
\end{center}
\caption{Plots of the Cauchy stress tensor components for the Hencky (OBI-H), OBI-P, and OBI-M models (see Table \ref{t4}) under LFSS (\emph{a}) and RFSS (\emph{b}) deformations.}
\label{f8}
\end{figure}

The markers in Figs. \ref{f7} and \ref{f8} show the results of computer simulations for the Ogden-A,B models (Fig. \ref{f7}) and the OBI-M model (Fig. \ref{f8}) using the homemade Pioner FE system provided by the compressible Mooney--Rivlin material model with the updated Lagrangian formulation based on the Biezeno--Hencky stress rate (see Appendix \ref{sec:A}). The markers for the Hencky (OBI-H) model in Fig. \ref{f8} are the same as in Fig. \ref{f6} (i.e., these markers correspond to computer simulations using the commercial MSC.Marc FE system).

\subsection{Hooke-like isotropic hypoelastic material models}
\label{sec:4-3}

Since the Hypo-A and Hypo-B hypoelastic material models (see Table \ref{t5}) based on the upper and lower Oldroyd stress rates with constitutive relations \eqref{2-47b} and \eqref{2-48} under isochoric deformation  without initial stresses are rate forms of hyperelastic material models with constitutive relations \eqref{4-14} and \eqref{4-15}, respectively (cf., \cite{KorobeynikovJElast2021}) and were tested above in Section \ref{sec:4-2}, we do not repeat the solutions of these problems and focus on testing the Hypo-ZJ, Hypo-GN, Hypo-GS, and Hypo-log Hooke-like isotropic hypoelastic material models (see Table \ref{t5}) based on corotational stress rates.

For hypoelastic materials, unlike hyperelasticity models, the assertions of Proposition \ref{Pr:3-2} do not hold in this section; therefore, the Eulerian version of constitutive relations for hypoelastic materials is tested under LFSS deformation and their Lagrangian version is then tested under RFSS deformation.

Under isochoric deformations, constitutive relations for Hooke-like isotropic hypoelastic material models \eqref{2-49} reduce to the constitutive relations
\begin{equation}\label{4-34}
    \bar{\boldsymbol{\sigma}}^{\Omega} = 2\mu\, \widehat{\mathbf{D}} \quad \Leftrightarrow \quad \boldsymbol{\sigma}^{\omega} = 2\mu\, \mathbf{d}
\end{equation}
in Lagrangian and Eulerian versions, respectively. Here $\bar{\boldsymbol{\sigma}}^{\Omega}$ and $\boldsymbol{\sigma}^{\omega}$ are corotational rates of the Lagrangian tensor $\bar{\boldsymbol{\sigma}}$ and the Eulerian tensor $\boldsymbol{\sigma}$, related by rotation operations of the form \eqref{2-28}, so that (see \eqref{2-23} and \eqref{2-26})
\begin{equation}\label{4-35}
   \bar{\boldsymbol{\sigma}}^{\Omega}\equiv \dot{\bar{\boldsymbol{\sigma}}} + \bar{\boldsymbol{\sigma}}\cdot \boldsymbol{\Omega} - \boldsymbol{\Omega}\cdot \bar{\boldsymbol{\sigma}},\quad\quad
   \left(\frac{\text{D}^{\circ}}{\text{D}t}[\boldsymbol{\sigma}] = \right)\, \boldsymbol{\sigma}^{\omega}\equiv \dot{\boldsymbol{\sigma}} + \boldsymbol{\sigma}\cdot \boldsymbol{\omega} - \boldsymbol{\omega}\cdot \boldsymbol{\sigma};
\end{equation}
furthermore, the spin tensors $\boldsymbol{\Omega}$ and $\boldsymbol{\omega}$ are related by transformations \eqref{2-30}.

We specify constitutive relations $\eqref{4-35}_2$ for LFSS deformations. Consider the representation of the spin tensor $\boldsymbol{\omega}$ in the form $\eqref{2-34}_2$. Since, for this type of deformation, all principal stretches of the tensor $\mathbf{V}_L$ are different for $t>t_0$ ($\lambda_1=\mathrm{e}^{\alpha}$, $\lambda_2=\mathrm{e}^{-\alpha}$), expression \eqref{2-35} in the 2D version, in view of the equality $g_{21}=-g_{12}$, reduces to the expression
\begin{equation}\label{4-36}
  \boldsymbol{\Psi}_g (\mathbf{V},\mathbf{d})= g_{12}(\mathbf{V}_1^L\cdot \mathbf{d}\cdot \mathbf{V}_2^L - \mathbf{V}_2^L\cdot \mathbf{d}\cdot \mathbf{V}_1^L),\quad\quad g_{12}\equiv g(\lambda_1,\lambda_2).
\end{equation}

We find expressions for the stretching and vorticity tensors for LFSS deformations. The tensor $\mathbf{F}_L^{-1}$ is obtained from $\eqref{3-8}_1$:
\begin{equation*}
  \mathbf{F}_L^{-1} = \frac{1}{\sqrt{\cosh (2\alpha)}}\left[
                                                   \begin{array}{cc}
                                                     \cosh (2\alpha) & -\sinh (2\alpha) \\
                                                            0        &       1          \\
                                                   \end{array}
                                                 \right].
\end{equation*}
The velocity gradient tensor is obtained from expression $\eqref{2-15}_2$. A chain of equalities leads to the expression
\begin{equation}\label{4-38}
  \boldsymbol{\ell} = \dot{\alpha}\left[
                        \begin{array}{cc}
                          -\gamma &      2 \\
                          0       & \gamma \\
                        \end{array}
                      \right],
\end{equation}
where
\begin{equation*}
  \gamma \equiv \tanh (2\alpha).
\end{equation*}
Using decompositions $\eqref{2-16}_1$ of the tensor $\boldsymbol{\ell}$ into symmetric and skew-symmetric parts $\eqref{2-16}_{2,3}$, we obtain expressions for the stretching and vorticity tensors from \eqref{4-38}:
\begin{equation}\label{4-40}
  \mathbf{d} = \dot{\alpha}\left[
                        \begin{array}{cc}
                          -\gamma &      1 \\
                          1       & \gamma \\
                        \end{array}
                      \right],\quad\quad
\mathbf{w} = \dot{\alpha}\left[
                        \begin{array}{cc}
                          0  & 1  \\
                          -1 & 0  \\
                        \end{array}
                      \right].
\end{equation}
Using expression $\eqref{4-40}_1$ for the stretching tensor $\mathbf{d}$ and expressions \eqref{3-14} for the eigenprojections $\mathbf{V}_1^L$ and $\mathbf{V}_2^L$, from  $\eqref{4-36}_1$ we obtain the following expression for the tensor $\boldsymbol{\Psi}_g$:
\begin{equation}\label{4-41}
  \boldsymbol{\Psi}_g = \dot{\alpha}\,g_{12}\,\gamma\left[
                        \begin{array}{cc}
                          0       &      1 \\
                          -1      &      0 \\
                        \end{array}
                      \right].
\end{equation}
Using Eqs. $\eqref{2-34}_2$, $\eqref{4-40}_2$, and \eqref{4-41}, we obtain the following  expression for any spin tensor $\boldsymbol{\omega}$:
\begin{equation}\label{4-42}
  \boldsymbol{\omega} = \dot{\alpha}\,(1+g_{12}\,\gamma)\left[
                        \begin{array}{cc}
                          0       &      1 \\
                          -1      &      0 \\
                        \end{array}
                      \right].
\end{equation}

In view of expressions $\eqref{4-35}_2$, $\eqref{4-40}_1$, and \eqref{4-42},   constitutive relations $\eqref{4-34}_2$ can be rewritten in component form
\begin{alignat}{2}\label{4-43}
  \dot{\sigma}_{11} - 2 \dot{\alpha}\,(1 + g_{12}\,\gamma)\,\sigma_{12}  & = & - &2\mu\,\dot{\alpha}\,\gamma,  \\
  \dot{\sigma}_{22} + 2 \dot{\alpha}\,(1 + g_{12}\,\gamma)\,\sigma_{12}  &  = &  &2\mu\, \dot{\alpha}\,\gamma,  \notag \\
  \dot{\sigma}_{12} +  \dot{\alpha}\,(1 + g_{12}\,\gamma)\,(\sigma_{11} - \sigma_{22})  & = & &2\mu\, \dot{\alpha}. \notag
\end{alignat}
Setting $\sigma_{ij}=0$ ($i,\,j=1,2$) for $t=t_0$, from $\eqref{4-43}_{1,2}$ we obtain the  equality
\begin{equation}\label{4-44}
   \sigma_{22}=- \sigma_{11}
\end{equation}
Using this relation, the system of ODEs \eqref{4-43} is reduced to the following system of two first-order ODEs:
\begin{alignat}{2}\label{4-45}
  \dot{\sigma}_{11} - 2 \dot{\alpha}\,(1 + g_{12}\,\gamma)\,\sigma_{12} &  = & -&2\mu\, \dot{\alpha}\,\gamma,  \\
  \dot{\sigma}_{12} +  2\dot{\alpha}\,(1 + g_{12}\,\gamma)\,\sigma_{11} & =  &  &2\mu\, \dot{\alpha}. \notag
\end{alignat}
Considering the shear parameter $\alpha$ as a monotonically increasing deformation parameter $t$ (i.e., setting $\dot{\alpha}=1$), we rewrite the system of ODEs \eqref{4-45} in the standard form
\begin{equation}\label{4-46}
  \frac{d \mathbf{X}}{d \alpha} + \mathbf{A}(\alpha)\mathbf{X} = \mathbf{G}(\alpha)
\end{equation}
with the initial conditions
\begin{equation}\label{4-47}
  \mathbf{X}=\mathbf{0}\quad \text{at}\quad \alpha=0.
\end{equation}
Here we introduced the notation
\begin{equation}\label{4-48}
  \mathbf{X}\equiv \left[
                     \begin{array}{c}
                       \sigma_{11} \\
                       \sigma_{12} \\
                     \end{array}
                   \right],\quad
   \mathbf{A}(\alpha)\equiv 2k(\alpha)\left[
                      \begin{array}{cc}
                        0   & -1 \\
                        1   & 0           \\
                      \end{array}
                    \right],\quad
   \mathbf{G}(\alpha)\equiv 2\mu \left[
                     \begin{array}{c}
                       \gamma(\alpha) \\
                       1 \\
                     \end{array}
                   \right],\quad
   k(\alpha)\equiv 1 + g_{12}(\alpha)\gamma(\alpha).
\end{equation}

To solve the system of ODEs \eqref{4-46} with the initial conditions \eqref{4-47},  we need to specify the quantity $g_{12}(\alpha)$ for each of the hypoelastic models considered in this paper using expressions \eqref{2-37}--\eqref{2-40}:
\begin{itemize}
  \item for the Hypo-ZJ model based on the Zaremba--Jaumann stress rate,
\begin{equation}\label{4-49}
  g_{12}^{\text{ZJ}} = 0\quad \Rightarrow \quad k^{\text{ZJ}}(\alpha)=1,
\end{equation}
  \item for the Hypo-GN model based on the Green--Naghdi stress rate,
\begin{equation}\label{4-50}
  g_{12}^{\text{GN}} = - \tanh (\alpha)\quad \Rightarrow \quad k^{\text{GN}}(\alpha) = \cosh^{-1}(2 \alpha),
\end{equation}
  \item for the Hypo-GS model based on the Gurtin--Spear stress rate associated with the twirl tensor of Eulerian triad,
\begin{equation}\label{4-51}
    g_{12}^{\text{GS}} = -\coth (2 \alpha)\quad \Rightarrow \quad k^{\text{GS}}(\alpha) = 0,
\end{equation}
  \item for the Hypo-log model based on the logarithmic stress rate,
\begin{equation}\label{4-52}
    g_{12}^{\log} = \frac{1}{2 \alpha}\quad \Rightarrow \quad k^{\log}(\alpha) = \frac{1}{2 \alpha}\tanh (2 \alpha).
\end{equation}
\end{itemize}

For the Hypo-ZJ and Hypo-GN models, the system of ODEs \eqref{4-46} with the initial conditions \eqref{4-47} was solved using the Wolfram Mathematica system. Despite the apparent simplicity of system \eqref{4-46}, the obtained solutions are rather cumbersome and are therefore not presented here. Plots of $\sigma_{12}(\alpha)$ and $\sigma_{11}(\alpha)$ obtained for the Hypo-ZJ and Hypo-GN models under LFSS deformation are shown in Figs. \ref{f9} and \ref{f10}, respectively.
\begin{figure}
\begin{center}
\includegraphics{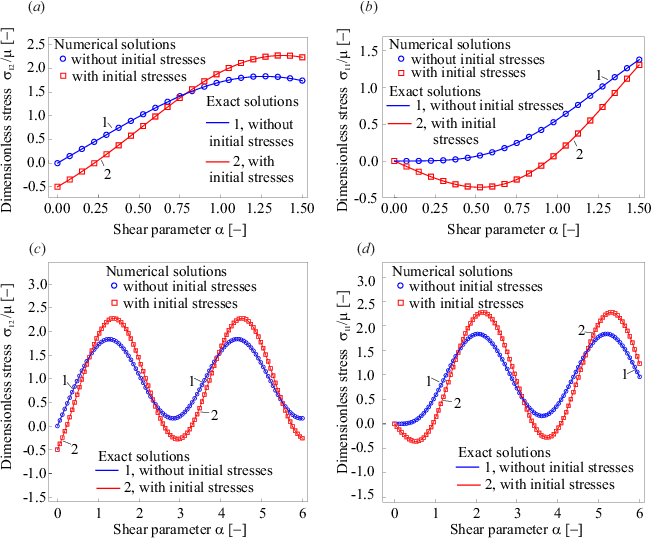}
\end{center}
\caption{Plots of the Cauchy stress tensor components $\sigma_{12}(\alpha)$ (\emph{a}, \emph{c}) and $\sigma_{11}(\alpha)$ (\emph{b}, \emph{d}) for the Hypo-ZJ model (see Table \ref{t5}) under LFSS deformation for the ranges of the shear parameter $\alpha$ [0,\,1.5] (\emph{a}, \emph{b}) and [0,\,6] (\emph{c}, \emph{d}).}
\label{f9}
\end{figure}
\begin{figure}
\begin{center}
\includegraphics{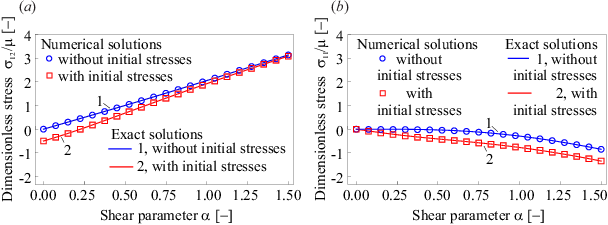}
\end{center}
\caption{Plots of the Cauchy stress tensor components $\sigma_{12}(\alpha)$ (\emph{a}) and $\sigma_{11}(\alpha)$ (\emph{b}) for the Hypo-GN model (see Table \ref{t5}) under LFSS deformation.}
\label{f10}
\end{figure}
Since the oscillatory behavior of the Cauchy stress tensor components for the Hypo-ZJ model is not sufficiently manifested in the standard range of the shear parameter $\alpha \in [0,\,1.5]$ (see Fig. \ref{f9},\emph{a,b}), the simulations were repeated over the wider range of this parameter $\alpha \in [0,\,6]$ (see Fig. \ref{f9},\emph{c,d}), where the mechanically meaningless oscillations of the Cauchy stress tensor components are already pronounced. Note that this type of oscillations under shear deformation is typical of this grade zero material model.\footnote{Of course, if instead of the model under consideration, based on the use of the linear dependence of the corotational Zaremba--Jaumann stress rate on the stretching tensor, i.e., $\frac{\text{D}^{\text{ZJ}}}{\text{D}t}[\boldsymbol{\tau}] = 2\mu\, \mathbf{d} + \lambda\, (\text{tr}\,\mathbf{d})\,\mathbf{I}$ we use a more complicated model in the form of a quasi-linear dependence of this corotational stress rate on the stretching tensor, i.e., $\frac{\text{D}^{\text{ZJ}}}{\text{D}t}[\boldsymbol{\tau}] = \mathbb{H}(\boldsymbol{\tau}):\mathbf{d}$ ($\mathbb{H}$ is some fourth-order tensor), for example, if these constitutive relations are the rate form of the hyperelasticity ones, then such spurious oscillations of stresses can be completely avoided.} In particular, similar oscillatory behavior of the stress rate components has been observed when solving a simple shear problem for the Hypo-ZJ (grade zero) model (see, e.g., \cite{BoulangerARMA2000,DestradeIJNLM2012,DienesAM1979,DienesAM1987,HorganJElast2010,GambirasioACME2016,LinIJNME2002,LinEJMAS2003,LinZAMM2003,LiuJElast1999,MihaiPRSA2011,MihaiIJNLM2013,Prager1961,SzaboIJSS1989,XiaoZAMM2006}).
At the same time, the Cauchy stress tensor components obtained for the Hypo-GN model under LFSS deformation, as in the solution of the simple shear problem (see, e.g., \cite{DienesAM1979,DienesAM1987}), do not exhibit any oscillations (see Fig. \ref{f10}).

For the hypoelastic model based on the Gurtin--Spear stress rate under LFSS deformation, it follows from \eqref{4-51} that the system of ODEs \eqref{4-45} splits into two uncoupled  equations for the quantities $\sigma_{12}(\alpha)$ and $\sigma_{11}(\alpha)$. The solutions of these equations with the initial conditions \eqref{4-47} are shown in Fig. \ref{f11}.
\begin{figure}
\begin{center}
\includegraphics{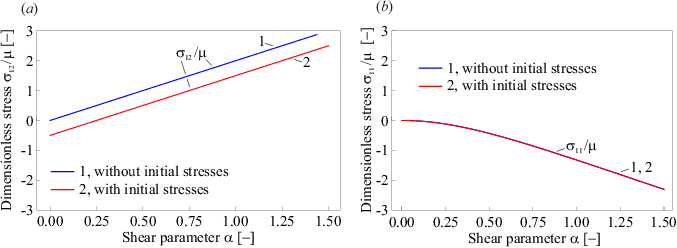}
\end{center}
\caption{Plots of the Cauchy stress tensor components $\sigma_{12}(\alpha)$ (\emph{a}) and $\sigma_{11}(\alpha)$ (\emph{b}) for the Hypo-GS model (see Table \ref{t5}) under LFSS deformation.}
\label{f11}
\end{figure}
It follows from the solution that the Cauchy stress tensor components obtained for the Hypo-GS model, under LFSS deformation, as in the solution of the same problem for the Hypo-GN model, do not exhibit  any oscillations.

Finding the Cauchy stress tensor components for the Hypo-log model with initial stresses \eqref{4-47} under LFSS deformation does not require the integration of the system of ODEs \eqref{4-46}, since in this case this Hooke-like (grade zero) hypoelastic model is equivalent to the Hencky isotropic hyperelastic material model \cite{XiaoAM1997,XiaoJElast1997,XiaoAM1999,XiaoJElast1999} and plots of the Cauchy stress tensor components  (without oscillations) obtained for this material model under LFSS deformation are already presented in Fig. \ref{f6}. Nevertheless, we also show these plots in Fig. \ref{f12}, since we will further compare the solutions of this problem with solutions obtained with non-zero initial stresses.
\begin{figure}
\begin{center}
\includegraphics{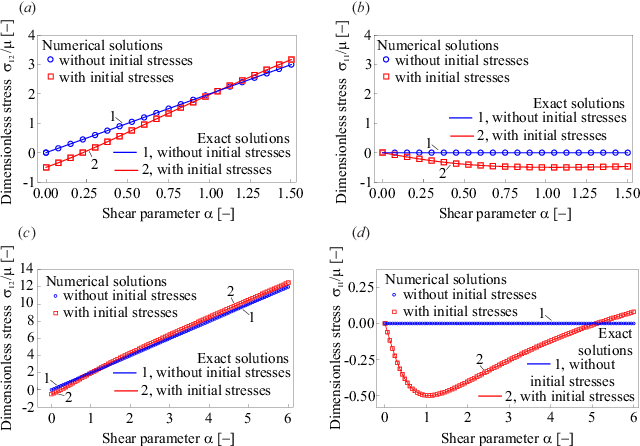}
\end{center}
\caption{Plots of the Cauchy stress tensor components $\sigma_{12}(\alpha)$ (\emph{a}, \emph{c}) and $\sigma_{11}(\alpha)$ (\emph{b}, \emph{d}) for the Hypo-log model (see Table \ref{t5}) under LFSS deformation for the ranges of the shear parameter $\alpha$ [0,\,1.5] (\emph{a}), (\emph{b}) and [0,\,6] (\emph{c}), (\emph{d}).}
\label{f12}
\end{figure}
It is known (cf., \cite{LiuJElast1999}, see also \cite{KorobeynikovZAMM2024}) that the Hypo-log model under simple shear deformation leads to the behavior of the Cauchy stress tensor components as a function of the shear parameter without oscillations in the case of zero initial stresses and with oscillations (mechanically meaningless) in the case non-zero initial stresses. The presence or absence of these features in the behavior of the Cauchy stress tensor components as a function of the shear parameter for the considered material model with non-zero initial stresses under LFSS deformation must be verified, and this will be done below.

We obtain expressions for the Cauchy stress tensor components for hypoelastic models based on the considered corotational stress rates under LFSS deformation in the presence of non-zero initial stresses. Following \cite{KorobeynikovZAMM2024}, we represent the solution of problem $\eqref{4-34}_2$ for the Cauchy stress tensor in the form
\begin{equation*}
  \boldsymbol{\sigma}=\boldsymbol{\sigma}_a + \boldsymbol{\sigma}_b,
\end{equation*}
where $\boldsymbol{\sigma}_a$ are the solutions of the homogeneous Cauchy problems with non-zero initial conditions
\begin{equation}\label{4-54}
    \boldsymbol{\sigma}^{\omega}_a= \mathbf{O},\quad\quad \boldsymbol{\sigma}_a=\boldsymbol{\sigma}^0\ \text{at}\ t=0,
\end{equation}
and $\boldsymbol{\sigma}_b$ are the solutions of the inhomogeneous Cauchy problems with zero initial conditions
\begin{equation}\label{4-55}
    \boldsymbol{\sigma}^{\omega}_b= 2\mu\, \mathbf{d},\quad\quad \boldsymbol{\sigma}_b=\mathbf{O}\ \text{at}\ t=0.
\end{equation}
The solutions of problems \eqref{4-55} for all hypoelastic models based on associated spin tensors considered in this paper were obtained above.

The solution of problem \eqref{4-54} for a hypoelastic model based on any spin tensor $\boldsymbol{\omega}$ can be written as \cite{KorobeynikovZAMM2024}
\begin{equation*}
  \boldsymbol{\sigma}_a=\mathbf{Q}_{\omega}\cdot\boldsymbol{\sigma}^0\cdot\mathbf{Q}_{\omega}^T,
\end{equation*}
where for 2D analysis, the rotation tensor $\mathbf{Q}_{\omega}\in \mathcal{T}^{2\,+}_\text{orth}$ can be written as
\begin{equation}\label{4-57}
\mathbf{Q}_{\omega}=
\left[
\begin{array}{cc}
 \cos\theta  & \sin\theta  \\
 -\sin\theta & \cos\theta
 \end{array}
\right].
\end{equation}
where $\theta$ is the rotation angle of the axes associated with the spin tensor $\boldsymbol{\omega}$.

In this paper, we adopt the following values for the initial stresses in 2D problems:
\begin{equation}\label{4-58}
    \sigma_{11}^0 = \sigma_{22}^0=0,\quad\quad\quad \sigma_{12}^0 = -\mu/2.
\end{equation}
Then the expressions presented in \cite{KorobeynikovZAMM2024} (p. 18) lead to the  following expressions for the components of the Cauchy stress tensor:
\begin{equation*}
  \sigma_{11}^a = \sigma_{12}^0\sin 2\theta,\quad\quad\quad \sigma_{22}^a = -\sigma_{12}^0\sin 2\theta,\quad\quad\quad \sigma_{12}^a = \sigma_{12}^0\cos 2\theta.
\end{equation*}
Note that, for initial stresses of the form \eqref{4-58}, equality \eqref{4-44} holds; therefore, below we present only solutions for the components $\sigma_{11}$ and $\sigma_{12}$ of the Cauchy stress tensor.

To complete the solution of the LFSS deformation problem, we need to express the angle $\theta$ in terms of the shear parameter $\alpha$ for each of the material models by solving the problem (see Eq. (31) in \cite{KorobeynikovZAMM2024})
\begin{equation}\label{4-60}
    \dot{\mathbf{Q}}_{\omega}= \boldsymbol{\omega}\cdot\mathbf{Q}_{\omega},\quad\quad\quad  \mathbf{Q}_{\omega}=\mathbf{I}\ \text{at}\ t=t_0.
\end{equation}
Substitution of \eqref{4-57} into $\eqref{4-60}_1$ leads to the following expression for the spin tensor $\boldsymbol{\omega}$
\begin{equation*}
  \boldsymbol{\omega} = \dot{\theta}\left[
                        \begin{array}{cc}
                          0       &      1 \\
                          -1      &      0 \\
                        \end{array}
                      \right].
\end{equation*}
On the other hand, the spin tensor $\boldsymbol{\omega}$ for LFSS deformation has the form \eqref{4-42}, which leads to the following equation for $\theta$:
\begin{equation}\label{4-62}
  \dot{\theta} = \dot{\alpha}[1+g_{12}(\alpha)\,\gamma(\alpha)].
\end{equation}
Setting $\alpha\equiv t-t_0$ and using the notation $\eqref{4-48}_4$, from \eqref{4-62} we obtain the Cauchy problem
\begin{equation*}
    \frac{d \theta}{d \alpha} = k(\alpha),\quad\quad  \theta=0\ \text{at}\ \alpha=0\ (\text{i.e.}, \text{at}\ t=t_0).
\end{equation*}

Using the expressions for the parameter $k(\alpha)$ in \eqref{4-49}--\eqref{4-52} for the  hypoelastic models considered, we obtain the following expressions for the angle $\theta$ for each of these models:
\begin{itemize}
  \item  For the Zaremba--Jaumann stress rate,
\begin{equation*}
  \theta(\alpha) = \alpha.
\end{equation*}
  \item For the Green--Naghdi stress rate (the expression is obtained using the Wolfram Mathematica system),
\begin{equation*}
  \theta(\alpha) = \text{arc}\tanh (\alpha),
\end{equation*}
i.e., $\theta=\theta^{\ast}$ (see Fig. \ref{f4},\emph{a}).
  \item For the Gurtin--Spear stress rate associated with the twirl tensor of Eulerian triad,
\begin{equation*}
    \theta(\alpha) = 0.
\end{equation*}
Since for this material model, the angle $\theta$ corresponds to the angle of rotation of the principal axes and these axes do not rotate, their rotation angle is zero.
  \item For the logarithmic stress rate, the angle $\theta$ is given by
\begin{equation*}
    \theta(\alpha) = \int_{0}^{\alpha} \frac{\tanh (2 \beta)}{2 \beta}\,d\beta.
\end{equation*}
Using the Wolfram Mathematica system, it was not possible to obtain an exact expression for this integral; therefore, the value of the angle $\theta$ for this material model will be obtained numerically using the same software.
\end{itemize}

The obtained solutions for the tested material models with initial stresses are presented in Figs. \ref{f9}--\ref{f12}. We see that the solution changes qualitatively (compared to the solution with zero initial stresses) only for the Hypo-log model, for which the presence of non-zero initial stresses leads to oscillations of the Cauchy stress tensor components (especially for the mechanically meaningless component $\sigma_{11}$). A similar qualitative change in the solution is observed for the simple shear problem in \cite{LiuJElast1999} (see also \cite{KorobeynikovZAMM2024}), where, for zero initial stresses, the Cauchy stress tensor components coincides with the solution of this problem for the Hencky isotropic hyperelastic material without oscillations, and the presence of non-zero initial stresses leads to mechanically meaningless oscillations of the Cauchy stress tensor components.

For greater reliability of the obtained solutions for the hypoelastic material models considered, we performed computer simulations of LFSS deformations for these materials using the homemade Pioner FE system \cite{Korobeinikov1989}, in which the Hooke-like Hypo-ZJ, Hypo-GN, and Hypo-log models were implemented. To integrate constitutive relations for any hypoelastic models, we used the weak incrementally objective algorithm of Rubinstein--Atluri with second-order accuracy \cite{RubinsteinCMAME1983} (see also \cite{Korobeynikov2023,KorobeynikovZAMM2024}). The obtained values of the Cauchy stress tensor components are represented by markers in Figs. \ref{f9}, \ref{f10}, and \ref{f12}. There is excellent agreement between the exact solutions and computer-simulated solutions.

Our next goal is to test Hooke-like isotropic hypoelastic models under RFSS deformation. Since this type of deformation is ``friendly'' to Lagrangian versions of constitutive relations, we further use Lagrangian version of ones for the Hooke-like isotropic hypoelastic models $\eqref{4-34}_1$ based on the corotational stress rates $\eqref{4-35}_1$ associated with the spin tensors $\boldsymbol{\Omega}$ determined from expressions $\eqref{2-31}_1$ and \eqref{2-32} adapted to the 2D analysis in the form
\begin{equation}\label{4-68}
    \boldsymbol{\Psi}_r(\mathbf{U},\widehat{\mathbf{D}}) = r_{12} (\mathbf{U}_1^R \cdot \widehat{\mathbf{D}} \cdot \mathbf{U}_2^R - \mathbf{U}_2^R \cdot \widehat{\mathbf{D}} \cdot \mathbf{U}_1^R),\quad\quad   r_{12} \equiv r(\lambda_1,\lambda_2),
\end{equation}
where the eigenprojections $\mathbf{U}_1^R$ and $\mathbf{U}_2^R$ are defined in \eqref{3-14}.\\

\begin{remark}
\label{rem:4-1}
The twirl tensor of Lagrangian triad is also included in the considered  family of material spin tensors $\boldsymbol{\Omega}$, since, for RFSS deformations, $\lambda_1\neq \lambda_2$ at $t>t_0$, which allows this spin tensor to be uniquely determined (see, e.g., \cite{KorobeynikovAM2011}).
\end{remark}

The Lagrangian stretching tensor is determined from expression $\eqref{2-18}_2$
\begin{equation}\label{4-69}
  \widehat{\mathbf{D}} = \frac{1}{2}(\dot{\mathbf{U}}_R \cdot \mathbf{U}^{-1}_R + \mathbf{U}^{-1}_R \cdot \dot{\mathbf{U}}_R).
\end{equation}
Here the tensors $\mathbf{U}_R$ and $\mathbf{U}^{-1}_R$ are determined from the expressions (see Eq. $\eqref{3-13}_2$)
\begin{equation}\label{4-70}
  \mathbf{U}_R = \lambda_1 \mathbf{U}_1^R + \lambda_2 \mathbf{U}_2^R,\quad\quad \mathbf{U}^{-1}_R = \lambda_1^{-1} \mathbf{U}_1^R + \lambda_2^{-1} \mathbf{U}_2^R,
\end{equation}
where $\lambda_1$ and $\lambda_2$ are defined in $\eqref{3-12}_{1,2}$. These expressions lead to
\begin{equation}\label{4-71}
  \dot{\lambda}_1 = \dot{\alpha}\mathrm{e}^{\alpha},\quad\quad \dot{\lambda}_2 = -\dot{\alpha}\mathrm{e}^{-\alpha},\quad\quad \lambda_1^{-1} = \mathrm{e}^{-\alpha},\quad\quad \lambda_2^{-1} = \mathrm{e}^{\alpha}.
\end{equation}
Since the eigenprojections $\mathbf{U}_1^R$ and $\mathbf{U}_2^R$ are independent on time $t$, the tensor $\dot{\mathbf{U}}_R$ is given by the expression
\begin{equation}\label{4-72}
  \dot{\mathbf{U}}_R = \dot{\lambda}_1 \mathbf{U}_1^R + \dot{\lambda}_2 \mathbf{U}_2^R.
\end{equation}
Using equalities \eqref{4-70}--\eqref{4-72}, from \eqref{4-69} we obtain the following  expression for the Lagrangian stretching tensor for RFSS deformations:
\begin{equation}\label{4-73}
  \widehat{\mathbf{D}} = \dot{\alpha}(\mathbf{U}_1^R - \mathbf{U}_2^R) = \dot{\alpha}\left[
                                                                             \begin{array}{cc}
                                                                               0 & 1 \\
                                                                               1 & 0 \\
                                                                             \end{array}
                                                                           \right].
\end{equation}
Using the expressions for $\mathbf{U}_1^R$ and $\mathbf{U}_2^R$ in \eqref{3-14} and the expression for the tensor $\widehat{\mathbf{D}}$ in \eqref{4-73}, we obtain the equalities
\begin{equation*}
  \mathbf{U}_1^R \cdot \widehat{\mathbf{D}} \cdot \mathbf{U}_2^R = \mathbf{U}_2^R \cdot \widehat{\mathbf{D}} \cdot \mathbf{U}_1^R = \left[
                                                                                                                  \begin{array}{cc}
                                                                                                                    0 & 0 \\
                                                                                                                    0 & 0 \\
                                                                                                                  \end{array}
                                                                                                                \right],
\end{equation*}
and from \eqref{4-68} we find that for any admissible $r_{12}$ and for any $\lambda_1,\,\lambda_1\in \mathds{R}^{+}$, the following equality holds:
\begin{equation*}
  \boldsymbol{\Psi}_r(\mathbf{U},\widehat{\mathbf{D}}) = \mathbf{O}\quad \Rightarrow\quad \boldsymbol{\Omega}= \mathbf{O}\quad \forall\ t>t_0.
\end{equation*}
Then it follows from $\eqref{4-35}_1$ that the family of all hypoelastic models with constitutive relations $\eqref{4-34}_1$ under RFSS deformation reduce to a single model with the constitutive relations
\begin{equation}\label{4-76}
    \dot{\bar{\boldsymbol{\sigma}}} = 2\mu\, \widehat{\mathbf{D}}.
\end{equation}

Writing both sides of expression \eqref{4-76} in matrix form and using expression \eqref{4-73}, we have
\begin{equation}\label{4-77}
  \left[
    \begin{array}{cc}
      \dot{\bar{\sigma}}_{11} & \dot{\bar{\sigma}}_{12} \\
      \dot{\bar{\sigma}}_{12} & \dot{\bar{\sigma}}_{22} \\
    \end{array}
  \right] = 2 \mu\, \dot{\alpha} \left[
                                               \begin{array}{cc}
                                                    0 & 1 \\
                                                    1 & 0 \\
                                               \end{array}
                                               \right].
\end{equation}
Assuming that, at the initial time $t_0$,
\begin{equation*}
    \bar{\sigma}_{11} = \bar{\sigma}_{22} = \bar{\sigma}_{12} = 0,
\end{equation*}
and setting $t-t_0=\alpha$, we obtain a solution of the system of ODEs \eqref{4-77} of the form
\begin{equation}\label{4-79}
    \bar{\sigma}_{11} = \bar{\sigma}_{22} = 0,\quad\quad \bar{\sigma}_{12} = 2\mu \alpha.
\end{equation}

We have shown that for all hypoelastic material models considered, RFSS deformations lead to Lagrangian pure shear stresses. Comparison of solution \eqref{4-79} with the solution obtained by testing the Hencky isotropic hyperelastic material model and presented in expressions \eqref{4-7} and \eqref{4-11} shows that they are identical. This agrees with the fact that solution \eqref{4-79} is independent of the deformation path, and, in the family of Hooke-like hypoelastic models based on corotational stress rates, the only model whose constitutive relations are the rate form of ones for hyperelasticity is the Hooke-like (grade zero) model based on the logarithmic stress rate (the Hypo-log model), whose constitutive relations are the rate form of ones for the Hencky hyperelastic material model provided that the reference configuration is the stress free one \cite{XiaoJElast1997,XiaoAM1999,XiaoJElast1999}. Since, for the Hencky model under RFSS deformation, the Cauchy stresses are given by \eqref{4-13} subject to \eqref{4-7}, it follows that for all hypoelastic models considered here, the Cauchy stress tensor for RFSS deformations is determined from the indicated expressions.

Note that our method of determining the components of the Cauchy stress tensor for the considered hypoelastic models under RFSS deformation using Lagrangian versions of constitutive relations is much simpler than the method proposed by Lin \cite{LinZAMP2024}. Lin used Eulerian versions of constitutive relations and, in addition, integrated ones separately for each material model (based on the Zaremba--Jaumann, Green--Naghdi, Gurtin--Spears, and logarithmic stress rates), whereas we have shown that constitutive relations for all material models under this deformation reduce to the rate form of ones for the Hencky material model. A direct comparison of our solutions for this type of constitutive relations with the solutions obtained by Lin is difficult, since we use the parameter $\alpha$ (as it is denoted in \cite{ThielIJNLM2019}) as the shear parameter, whereas in \cite{LinZAMP2024} hypoelastic models are tested only under RFSS deformation, and the shear parameter is taken to be the angle $\theta \equiv \arctan (\sinh (2\alpha))$, which in \cite{LinZAMP2024} is denoted  by $\alpha$. However, it can be shown that the solution for the Cauchy stress tensor components presented in Eq. (26) by Lin \cite{LinZAMP2024} coincides with our solution \eqref{4-13} and the solution for the rotated Cauchy stress tensor $\bar{\boldsymbol{\sigma}}$ in expression (28) obtained by Lin \cite{LinZAMP2024} coincides with our solution \eqref{4-79}.

\section{Conclusions}
\label{sec:8}

We have tested some Hooke-like isotropic hyper-/hypo-elastic material models under finite simple shear deformations and have shown the following:
\begin{enumerate}
  \item the components of the Cauchy stress tensor for any isotropic Cauchy/Green elastic material under left finite simple shear (LFSS) deformation are equal to the components of the rotated Cauchy stress tensor for the same material under right finite simple shear (RFSS) deformation;
  \item for any HLIH material model based on a symmetrically physical (SP) strain function, LFSS and RFSS deformations lead to Eulerian and Lagrangian pure shear stresses, respectively;
  \item for any two-exponential Ogden's isotropic hyperelastic material model based on a SP strain function, LFSS and RFSS deformations lead to Eulerian and Lagrangian pure shear stresses, respectively;
  \item for some isotropic Hooke-like hypoelastic materials with constitutive relations based on corotational stress rates under LFSS deformation, the evolution of the Cauchy stress tensor components
 as a function of the shear parameter is qualitatively similar to that for the same materials under simple shear deformation.
\end{enumerate}

In addition, we have confirmed the results of Lin \cite{LinZAMP2024} that the Cauchy stress tensor components for isotropic Hooke-like hypoelastic materials with constitutive relations based on any corotational stress rates without initial stresses under RFSS deformation coincide with those for the Hencky isotropic hyperelastic material.

\renewcommand{\theequation}{A.\arabic{equation}}
\setcounter{equation}{0}

\appendix

\section{Constitutive relations for the compressible isotropic hyperelastic Mooney-Rivlin material model}
\label{sec:A}

The incompressible isotropic hyperelastic Mooney-Rivlin material model can be considered as a special case of the incompressible isotropic hyperelastic Ogden model \cite{OgdenPRSLA1972a}, namely, the two-power Ogden model with elastic energy with terms corresponding to the exponents $n=\pm 2$. Ogden \cite{OgdenPRSLA1972b} generalized its model to account for material compressibility. For this, Ogden represented elastic energy in mixed form.\footnote{The mixed form of elastic energy is the sum of a purely volumetric part and a mixed (volumetric and isochoric) component; an alternative representation of volumetric energy is the vol-iso form with splitting into purely volumetric and isochoric parts \cite{Korobeynikov2026}.} Using this approach to formulate the elastic energy for compressible isotropic hyperelastic Mooney-Rivlin materials, we find, following \cite{KorobeynikovAAM2023}, that the potential energy for this material model is given by\footnote{In \eqref{A-1}, the first two terms on the right-hand side correspond to the mixed type of potential energy and the last term $h(J)\equiv \lambda (\log J)^2/2$ corresponds to the potential energy of pure volumetric deformation which is, however, not convex in $J$. Following \cite{OgdenPRSLA1972b} (see also \cite{HartmannIJSS2003,Korobeynikov2026}), we can use other expressions for this potential energy to satisfy some constraints on the  function $h(J)$.}
\begin{align}\label{A-1}
  W_{\text{MR}} \equiv {}& \mu_1(\text{tr}\,\mathbf{e}^{(2)}- \log J) - \mu_2(\text{tr}\,\mathbf{e}^{(-2)}-\log J) + \frac{\lambda}{2}(\log J)^2 \\
  = {}&  \frac{\mu_1}{2}\, (\|F\|^2 - 3 - 2\log J) + \frac{\mu_2}{2}\,(\|F^{-1}\|^2 - 3 + 2\log J) + \frac{\lambda}{2}(\log J)^2. \notag
\end{align}
Hereinafter, $\mu_1\geq 0$ and $\mu_2\geq 0$ are model parameters that satisfy the equality \linebreak $\mu_1+\mu_2=\mu$, $\lambda$ and $\mu$ are the Lam\'{e} parameters, and $\textbf{e}^{(2)} = (\mathbf{c}-\mathbf{I})/2$ and \linebreak $\textbf{e}^{(-2)} = (\mathbf{I}-\mathbf{c}^{-1})/2$ are the Finger and Almansi strain tensors (see Table \ref{t1}).

Since, for isotropic Cauchy/Green elastic materials, the Kirchhoff stress tensor $\boldsymbol{\tau}$ is coaxial with the left stretch tensor $\mathbf{V}$ (see, e.g., \cite{Ogden1984}), the tensor $\boldsymbol{\tau}$ can be represented as ($m$ is the eigenindex of the tensor $\mathbf{V}$)
\begin{equation}\label{A-2}
    \boldsymbol{\tau}=\sum^m_{i=1}\tau_i \mathbf{V}_i,
\end{equation}
where $\mathbf{V}_i$ ($i=1,\ldots ,m$) are the eigenprojections of the tensor $\mathbf{V}$. The components $\tau_i$ ($i=1,\ldots ,m$) can be found from the expressions (see, e.g., \cite{KorobeynikovJElast2019}, Eq. $(40)_2$)
\begin{equation}\label{A-3}
    \boldsymbol{\tau}=\sum^m_{i=1}\lambda_i\frac{\partial W_{\text{MR}}}{\partial \lambda_i} \mathbf{V}_i,
\end{equation}
where $\lambda_i$ ($i=1,\ldots ,m$) are the principal stretches. Basis-free expressions for the tensor $\boldsymbol{\tau}$ can be derived from \eqref{A-1}-\eqref{A-3}:
\begin{align}\label{A-4}
  \boldsymbol{\tau} = {}& 2 \mu_1\,\mathbf{e}^{(2)} + 2 \mu_2\,\mathbf{e}^{(-2)} + \lambda\,(\log J)\,\mathbf{I}\quad \Leftrightarrow \\
  \boldsymbol{\tau} = {}& \mu_1\,(\mathbf{c}-\mathbf{I}) + \mu_2\,(\mathbf{I}-\mathbf{c}^{-1}) + \lambda\,(\log J)\,\mathbf{I}. \notag
\end{align}

To obtain explicit expressions for the elasticity tensors, we define the \emph{dyadic} $\mathbf{A}\otimes\mathbf{H}$ and \emph{symmetric} $\mathbf{A}\!\overset{\text{sym}}{\otimes}\!\mathbf{H}$ tensor products for the second-order tensors $\mathbf{A}$ and $\mathbf{H}$ according to the definition given in \cite{Curnier1994}. Let $\mathbf{A}$, $\mathbf{B}$, and $\mathbf{X}$ be second-order tensors; then the following identities hold:
\begin{equation*}
     (\mathbf{A}\otimes\mathbf{B}):\mathbf{X} = \mathbf{A}\,(\mathbf{B}:\mathbf{X}),\quad
    (\mathbf{A}\!\overset{\text{sym}}{\otimes}\!\mathbf{B}):\mathbf{X} = \mathbf{A}\cdot \text{sym}\,\mathbf{X}\cdot\mathbf{B}^T,\quad
    \text{sym}\,\mathbf{X}\equiv (\mathbf{X}+\mathbf{X}^T)/2.
\end{equation*}

According to the method of obtaining objective rate forms of constitutive relations for the isotropic hyperelastic material models developed \cite{KorobeynikovAAM2025}, the rate form for the material model considered can be written as
\begin{equation}\label{A-6}
  \boldsymbol{\tau}^{\text{ZJ}} = \mathbf{d}\cdot (\mu_1\,\mathbf{c} + \mu_2\,\mathbf{c}^{-1}) + (\mu_1\,\mathbf{c} + \mu_2\,\mathbf{c}^{-1})\cdot \mathbf{d} + \lambda\,(\text{tr}\,\mathbf{d})\,\mathbf{I},
\end{equation}
or in alternative form
\begin{equation}\label{A-7}
  \boldsymbol{\tau}^{\text{ZJ}} = \mathbb{C}^{\text{ZJ}}_{\text{MR}}:\mathbf{d},
\end{equation}
where $\mathbb{C}^{\text{ZJ}}_{\text{MR}}$ is the fourth-order \emph{elasticity tensor} of the form
\begin{equation*}
  \mathbb{C}^{\text{ZJ}}_{\text{MR}}\equiv \mathbf{I}\!\overset{\text{sym}}{\otimes}\!(\mu_1\,\mathbf{c} + \mu_2\,\mathbf{c}^{-1}) + (\mu_1\,\mathbf{c} + \mu_2\,\mathbf{c}^{-1})\!\overset{\text{sym}}{\otimes}\!\mathbf{I} + \lambda\,\mathbf{I}\otimes \mathbf{I},
\end{equation*}
which for hyperelastic materials has full symmetry (or super-symmetry) \cite{FedericoMMS2025}.

We implemented the compressible isotropic hyperelastic Mooney-Rivlin material model with constitutive relations \eqref{A-4} and the super-symmetric elasticity tensor
\begin{equation*}
  \mathbb{C}^{\text{BH}}_{\text{MR}}\equiv \frac{1}{J}\mathbb{C}^{\text{ZJ}}_{\text{MR}}
\end{equation*}
in the homemade Pioner nonlinear FE system \cite{Korobeinikov1989} using an updated Lagrangian formulation based on the rate form of constitutive relations
\begin{equation*}
  \boldsymbol{\sigma}^{\text{BH}} = \mathbb{C}^{\text{BH}}_{\text{MR}}:\mathbf{d},
\end{equation*}
where $\boldsymbol{\sigma}^{\text{BH}}$ is the \emph{Biezeno--Hencky stress rate} \cite{Biezeno1928}
\begin{equation*}
  \boldsymbol{\sigma}^{\text{BH}} \equiv \boldsymbol{\sigma}^{\text{ZJ}} + \boldsymbol{\sigma} (\text{tr}\,\mathbf{d})
\end{equation*}
related to the stress rate $\boldsymbol{\tau}^{\text{ZJ}}$ as follows (see, e.g., \cite{Korobeynikov2026}):
\begin{equation*}
  \boldsymbol{\tau}^{\text{ZJ}} = J\,\boldsymbol{\sigma}^{\text{BH}}.
\end{equation*}

Next, we consider three well-known hyperelasticity models which are special cases of the compressible isotropic hyperelastic Mooney-Rivlin material model.\\

(i) The compressible mixed-type \emph{neo-Hookean} \cite{Korobeynikov2026} (or \emph{Simo--Pister} \cite{SimoCMAME1984}) material model with the following constraints on the model parameters: $\mu_1=\mu$ and $\mu_2=0$. In our classification of material models (see Table \ref{t3}), this model is called the \emph{Ogden-A isotropic compressible hyperelastic material model}. Constitutive relations for this material model are obtained from $\eqref{A-4}_2$:
\begin{equation}\label{A-13}
  \boldsymbol{\tau} = \mu\,(\mathbf{c}-\mathbf{I}) + \lambda\,(\log J)\,\mathbf{I}\ = \mu\,(B - I) + \lambda\,(\log J)\,I.
\end{equation}
Using the relationship between the objective Zaremba--Jaumann and upper Oldroyd stress rates
\begin{equation*}
  \boldsymbol{\tau}^{\text{ZJ}} = \boldsymbol{\tau}^{\sharp} + \boldsymbol{\tau}\cdot \mathbf{d} + \mathbf{d}\cdot \boldsymbol{\tau},
\end{equation*}
and the constitutive relations \eqref{A-13}, we obtain the following rate form of constitutive relations for the material model considered, which is alternative to \eqref{A-6}:
\begin{equation*}
  \boldsymbol{\tau}^{\sharp} = 2(\mu - \lambda\,\log J)\mathbf{d} + \lambda\,(\text{tr}\,\mathbf{d})\,\mathbf{I}.
\end{equation*}
This formulation is consistent with the rate form of constitutive relations (60) in \cite{KorobeynikovAAM2023} (for $n=2$) and (7.4) in \cite{Korobeynikov2026}.\\

(ii) The compressible mixed-type isotropic hyperelastic material model with the following constraints on the model parameters: $\mu_1=0$ and $\mu_2=\mu$. In our classification of material models (see Table \ref{t3}), this model is called the \emph{Ogden-B isotropic compressible hyperelastic material model}. Constitutive relations for this material model are derived from $\eqref{A-4}_2$:
\begin{equation}\label{A-16}
  \boldsymbol{\tau} = \mu\,(\mathbf{I}-\mathbf{c}^{-1}) + \lambda\,(\log J)\,\mathbf{I}\ =\mu\,(I - B^{-1}) + \lambda\,(\log J)\,I.
\end{equation}
Using the relationship between the objective Zaremba--Jaumann and lower Oldroyd stress rates
\begin{equation*}
  \boldsymbol{\tau}^{\text{ZJ}} = \boldsymbol{\tau}^{\flat} - \boldsymbol{\tau}\cdot \mathbf{d} - \mathbf{d}\cdot \boldsymbol{\tau},
\end{equation*}
and the constitutive relations \eqref{A-16}, we obtain the following rate form of constitutive relations for the material model considered, which is alternative to \eqref{A-6}:
\begin{equation*}
  \boldsymbol{\tau}^{\flat} = 2(\mu + \lambda\,\log J)\mathbf{d} + \lambda\,(\text{tr}\,\mathbf{d})\,\mathbf{I}.
\end{equation*}
This formulation is consistent with the rate form of constitutive relations (60) in \cite{KorobeynikovAAM2023} (for $n=-2$).\\

(iii) The compressible mixed-type isotropic hyperelastic material model with the following constraints on the model parameters: $\mu_1=\mu_2=\mu/2$. In our classification of material models (see Table \ref{t4}), this model is called the \emph{OBI-M isotropic compressible hyperelastic material model}. Constitutive relations for this material model are derived from $\eqref{A-4}_2$
\begin{align}\label{A-19}
  \boldsymbol{\tau} = \frac{\mu}{2}\,(\mathbf{c}-\mathbf{c}^{-1}) + \lambda\,(\log J)\,\mathbf{I}\ = {} & \frac{\mu}{2}\,(B - B^{-1}) + \lambda\,(\log J)\,I \\
   = {} &\frac{\mu}{2}\,(B - B^{-1}) + \frac{\lambda}{2}\,(\log\, \det B)\,I. \notag
\end{align}
For this material model, the rate form of constitutive relations \eqref{A-6} reduces to
\begin{equation}\label{A-20}
  \boldsymbol{\tau}^{\text{ZJ}} = \frac{\mu}{2}\,\mathbf{d}\cdot (\mathbf{c} + \mathbf{c}^{-1}) + \frac{\mu}{2}\,(\mathbf{c} + \mathbf{c}^{-1})\cdot \mathbf{d} + \lambda\,(\text{tr}\,\mathbf{d})\,\mathbf{I},
\end{equation}
or to formula \eqref{A-7}, which is alternative to  formula \eqref{A-20}, with the super-symmetric fourth-order elasticity tensor
\begin{equation*}
  \mathbb{C}^{\text{ZJ}}_{\text{MR}} = \frac{\mu}{2}\,\mathbf{I}\!\overset{\text{sym}}{\otimes}\!(\mathbf{c} + \mathbf{c}^{-1}) + \frac{\mu}{2}\,(\mathbf{c} + \mathbf{c}^{-1})\!\overset{\text{sym}}{\otimes}\!\mathbf{I} + \lambda\,\mathbf{I}\otimes \mathbf{I}.
\end{equation*}
Note that expressions \eqref{A-19} and \eqref{A-20} for constitutive relations for the hyperelastic material model considered here differ from expressions (2.80) and (2.81) in \cite{NeffJNS2026} for constitutive relations for the Cauchy elastic material model considered in that paper in that they use the Cauchy stress tensor $\boldsymbol{\sigma}$ instead of the Kirchhoff one $\boldsymbol{\tau}$ in relation \eqref{A-19}.

Based on \eqref{A-19}, we can give an alternative derivation of the rate form \eqref{A-20}. It holds using a chain rule property of corotational rates together with the gradient of isotropic tensor function \cite{NeffAM2025}
\begin{equation*}
  \frac{\mathrm{D}^{\mathrm{ZJ}}}{\mathrm{D}\,t}[\tau] = \mathrm{D}_B\,\tau(B).\frac{\mathrm{D}^{\mathrm{ZJ}}}{\mathrm{D}\,t}[B]\quad (\Leftrightarrow\quad \boldsymbol{\tau}^{\text{ZJ}} = \frac{\partial\,\boldsymbol{\tau}(\mathbf{c})}{\partial\,\mathbf{c}}:\mathbf{c}^{\text{ZJ}}).
\end{equation*}
The following equality holds (see Eq. (1.7) in \cite{NeffAM2025}):
\begin{equation*}
  \frac{\mathrm{D}^{\mathrm{ZJ}}}{\mathrm{D}\,t}[B] = B D + D B \quad (\Leftrightarrow\quad \mathbf{c}^{\text{ZJ}} = \mathbf{c}\cdot \mathbf{d} + \mathbf{d}\cdot \mathbf{c}).
\end{equation*}
Let $H\in \mathcal{T}^2$ be any tensor, then we have for the tensor $\tau$ from \eqref{A-19} the following expression (see Eq. (3.11) in \cite{NeffJMPS2025})
\begin{align*}
  \mathrm{D}_B\,\tau(B).H = {} & \frac{\mu}{2}\,(H + B^{-1}H B^{-1}) + \frac{\lambda}{2} \langle B^{-1},H \rangle I \\
  [\Leftrightarrow\quad\frac{\partial\,\boldsymbol{\tau}(\mathbf{c})}{\partial\,\mathbf{c}}:\mathbf{H}={}&\frac{\mu}{2}\,(\mathbf{H} +\mathbf{c}^{-1}\cdot\mathbf{H}\cdot\mathbf{c}^{-1}) + \frac{\lambda}{2}(\mathbf{c}^{-1}:\mathbf{H})\mathbf{I}].\notag
\end{align*}
Using this expression with the identification $H = BD + DB$, we have the final equality (see Eq. (3.12) in \cite{NeffJMPS2025})
\begin{equation*}
  \frac{\mathrm{D}^{\mathrm{ZJ}}}{\mathrm{D}\,t}[\tau] = \frac{\mu}{2}\,(BD + DB + B^{-1}D + D B^{-1}) + \lambda\,(\text{tr}\,D)\,I,
\end{equation*}
which is equivalent to Eq. \eqref{A-20}.

\footnotesize
\bibliography{KLN_final_07_02_2026}

\end{document}